\documentclass[twoside, 11pt]{amsart}

\usepackage[english]{babel}

\usepackage[T1]{fontenc}
\usepackage[lf]{Baskervaldx} 
\usepackage[bigdelims,vvarbb]{newtxmath} 
\usepackage[cal=boondoxo]{mathalfa} 

\usepackage{latexsym}
\usepackage{tensor}
\usepackage{mathrsfs}
\usepackage{times,amsmath,amsthm}
\usepackage{graphicx}
\usepackage{a4wide}
\usepackage[all]{xy}
\usepackage{paralist}
\usepackage{subfigure}
\usepackage{multicol}
\usepackage{tikz}
\usepackage{tikz-cd}
\usetikzlibrary{decorations.pathmorphing,arrows,arrows.meta,calc,shapes,shapes.geometric,decorations.pathreplacing, fit, positioning, matrix}
\usetikzlibrary{decorations.pathmorphing,decorations.markings,arrows,calc,shapes.geometric,arrows.meta,positioning}
\usepackage{comment}
\tikzset{
  optree/.style={scale=.5,thick,grow'=up,level distance=10mm,inner sep=1pt},
  comp/.style={draw=none,circle,fill,line width=0,inner sep=0pt},
  dot/.style={draw,circle,fill,inner sep=0pt,minimum width=3pt},
  circ/.style={draw,circle,inner sep=1pt,minimum width=4mm},
  emptycirc/.style={draw,circle,inner sep=1pt,minimum width=2mm},
  root/.style={level distance=10mm,inner sep=1pt},
  leaf/.style={draw=none,circle,fill,line width=0,inner sep=0pt},
  nodot/.style={draw,circle,inner sep=1pt},
}

\usepackage{color}
\definecolor{Chocolat}{rgb}{0.36, 0.2, 0.09}
\definecolor{BleuTresFonce}{rgb}{0.215, 0.215, 0.36}

\usepackage[colorlinks,final,hyperindex, ]{hyperref}
\hypersetup{citecolor=BleuTresFonce, linkcolor=Chocolat, urlcolor  = BleuTresFonce}
\usepackage[noabbrev,capitalize]{cleveref}

\usepackage[normalem]{ulem}

\usepackage{todonotes}

\let\oldtocsection=\tocsection
\let\oldtocsubsection=\tocsubsection

\renewcommand{\tocsection}[2]{\hspace{0em}\vspace{0.1em}\rule{0pt}{14pt}\oldtocsection{#1}{#2}\bf}
\renewcommand{\tocsubsection}[2]{\hspace{2em}\oldtocsubsection{#1}{#2}}

\newcommand{\PaPRT}{\ensuremath{\mathsf{PaPRT}}}
\newcommand{\R}{\ensuremath{\mathrm{R}}}
\def\Ho#1#2{\Lambda^{#2}_{#1}}
\def\De#1{\Delta^{#1}}
\newcommand{\sLi}{\ensuremath{\mathrm{sL}_\infty}}
\newcommand{\mc}{\ensuremath{\mathfrak{mc}}}
\newcommand{\hatCobar}{\ensuremath{\widehat{\Omega}}}
\newcommand{\sLialg}{\ensuremath{\mathrm{sL}_\infty\text{-}\,\mathsf{alg}}}
\newcommand{\sSe}{\mathsf{sSet}}
\def\Li{\mathrm{L}_\infty}

\def\Liegra{\mathrm{Lie}\textrm{-}\mathrm{graph}}

\def\ln{\mathrm{ln}}
\def\k{\mathbb{k}}

\def\F{\mathrm{F}}
\def\d{\mathrm{d}}
\def\P{\mathcal{P}}

\def\1{\mathbb{1}}

\def\id{\mathrm{id}}
\def\End{\mathrm{End}}

\def\gra{\mathrm{g}}
\newcommand{\Sy}{\mathbb{S}}
\def\dcGra{\mathsf{dsGra}}

\def\ldcGra{2\textsf{-}\mathsf{dsGra}}

\def\3ldcncGra{3\textsf{-}\mathsf{dsncGra}}

\def\kdcGra{\k\mathsf{dsGra}}
\def\btGra{{\Join}\textsf{-}{\dcGra}_{\Sy}}
\def\Lieadm{\mathrm{Lie}\textrm{-}\mathrm{adm}}
\def\Prelie{\mathrm{pre}\textrm{-}\mathrm{Lie}}
\def\Lie{\mathrm{Lie}}

\newcommand{\ac}{\scriptstyle \text{\rm !`}}
\def\G{\mathfrak{G}}
\def\g{\mathfrak{g}}

\def\Hom{\mathrm{Hom}}
\newcommand{\BCH}{\mathrm{BCH}}

\def\cc{\circledcirc}
\def\ad{\mathrm{ad}}
\def\MC{\mathrm{MC}}
\def\pap#1#2#3{#1\stackrel{#2}{\vcenter{\hbox{\text{\scalebox{1.2}{$\Join$}}}}}#3}

\def\Aut{\mathrm{Aut}}

\numberwithin{equation}{section}
\theoremstyle{plain}
\newtheorem{proposition}[equation]{Proposition}
\newtheorem{theorem}[equation]{Theorem}
\newtheorem{corollary}[equation]{Corollary}

\theoremstyle{definition}
\newtheorem{definition}[equation]{Definition}

\newtheorem{example}[equation]{\sc Example}

\newtheorem{application}[equation]{\sc Application}
\newtheorem{Historical}[equation]{\sc Historical remark}



\title{Effective integration of Lie type algebras}

\date{\today}

\author{Bruno Vallette}
\address{Universit\'e Sorbonne Paris Nord, Laboratoire de G\'eom\'etrie, Analyse et Applications, CNRS, UMR 7539, Villetaneuse, France}
\email{vallette@math.univ-paris13.fr}

\thanks{
\newline
The authors are supported by ANR-20-CE40-0016 HighAGT}

\begin{document}

\maketitle

\begin{abstract}
This is a short survey on the recent developments made in the integration theory with effective formulas of algebraic structures stronger or higher than Lie algebras.
\end{abstract}

\tableofcontents

\section*{Introduction}
In Lie theory, the local structure of a Lie group is faithfully encoded into its tangent Lie algebra. 
Lie's third theorem goes the other way round and tells us how to integrate finite dimensional real Lie algebras in order to obtain simply connected real Lie groups. 
The fundamental exponential map yields a group morphism from the Lie algebra equipped with the universal  Baker--Campbell--Hausdorff formula to the underlying Lie group. 

\medskip

In deformation theory, the underlying heuristic of the characteristic $0$ case says that in the study of the moduli spaces of a certain type of structures, there should exist a differential graded Lie algebra whose Maurer--Cartan elements should correspond to the various structures and where the gauge group action should encode the equivalences between them. Many examples were worked out like  complex manifolds (Kodaira--Spencer \cite{KodairaSpencer58}), associative algebras (M. Gerstenhaber \cite{Gerstenhaber64}), and Poisson manifolds (M. Kontsevich \cite{Kontsevich03}), before that this becomes a precise statement settled by J. Pridham \cite{pridham2010unifying} and J. Lurie \cite{Lurie10} using homotopical methods and recently by B. Le Grignou and V. Roca i Lucio \cite{LGRL23bis} using the operadic calculus. 

\medskip

The general formulas of deformation theory are rather intricate, both for the gauge group structure, produced by the Baker--Campbell--Hausdorff formula, and for its action on Maurer--Cartan elements. But, for Lie brackets that are obtained by skew-symmetrizing  non-symmetric products, one can apply the ideas of the classical Lie theory to produce another (topological) group related to the gauge group under an exponential map. This approach has the bright advantage to produce simpler explicit and effective formulas for the new ``exponential'' group and its action on Maurer--Cartan elements. After recalling the straightforward case of associative algebras, we treat in detail the cases of pre-Lie algebras and Lie-graph algebras, which is a new type of Lie type algebraic structure introduced to control to deformation theory of morphisms of properads.
\[\mathrm{associative} \subset \mathrm{pre}\textrm{-}\mathrm{Lie} \subset \mathrm{Lie}\textrm{-}\mathrm{graph} \xrightarrow{-}
\boxed{\Lie}  \subset \ \mathrm{L}_\infty \subset \mathrm{curved}\ \mathrm{L}_\infty \subset 
\mathrm{absolute\ curved} \ \mathrm{EL}_\infty\]
Dually, one can ask how to integrate effectively algebraic structures which are weaker versions of Lie algebras, like homotopy Lie algebras, where the Jacobi relation is not fully satisfied. The answer is actually produced by homotopy theory and it amonts to considering higher gauges. This higher Lie theory carries the advantages of producing algorithmic Lie type models for the rational and the $p$-adic homotopy type of spaces. 
\medskip

This short survey is by no means exhaustive as this subject is rich enough to form the topic of a whole book. 
For instance, we do not address here the interesting integration theories of Leibniz algebras or post-Lie algebras. 
The sole purpose of the present text is to offer the reader a glimpse of the global picture of this program. 
Like the celebrated effective formulas of Agrachev--Gamkrelidze \cite{AgrachevGamkrelidze80} and Bruned--Hairer--Zambotti \cite{BHZ19}, the results detailled here are expected to receive more applications in the future; we hope that the present survey will ease the way in this direction. 

\medskip

\paragraph*{\bf Convention}
Except in the last section, we work over a field $\k$ of characteristic $0$. Chain complexes are homologically graded, that is their differential maps have degree $-1$. 

\medskip

\paragraph*{\bf Acknowledgements}
I would like to express my appreciation to the referees for their wise comments which helped to improve the present text. 


\section{The classical Lie case}

We begin by recalling of the universal formulas for the deformation and the integration theories of differential graded Lie algebras. One way to get them amonts to considering the stronger case of differential graded associative algebras. 

\begin{definition}[Complete differential graded Lie algebra]
A \emph{complete differential graded (dg) Lie algebras} is a Lie algebra $\g=\left(A, \d, \F, [\,,]\right)$ in the symmetric monoidal category of chain complexes equipped with compatible and complete filtrations: 
\[A_n =\F_0 A_n \supset \F_1 A_n \supset \cdots \F_k A_n \supset \F_{k+1} A_n\supset \cdots 
\ , \quad  
A_n \cong \lim_{k\in \mathbb{N}} A_n/\F_k A_n
\quad   \& \quad  \d\left(\F_k A_n\right)\subset \F_k A_{n-1}~.\]
\end{definition}

\begin{definition}[Maurer--Cartan element]
A \emph{Maurer--Cartan element} of a complete dg Lie algebra $\g$ is a degree $-1$ element $\alpha$ satisfying the \emph{Maurer--Cartan equation}: 
\[\d\alpha + \tfrac12 [\alpha, \alpha]=0~.\]
\end{definition}

We denote by $\MC(\g)$ the set of Maurer-Cartan elements. 
When the space $\g_{-1}$ is finite dimensional, the Maurer--Cartan equation defines a finite intersection of quadrics, so the Maurer--Cartan set 
$\MC(\g)$ becomes a variety. 
The degree $0$ elements $\lambda \in \F_1A_0$~, called \emph{the gauges}, canonically give rise to a vector fields 
\[ -\d\lambda + \ad_\lambda \in \Gamma( \mathrm{T}\,  \mathrm{MC}(\g))\ ,\]
where the adjoint representation is $\ad_\lambda(x)\coloneq [\lambda, x]$~.

\begin{definition}[Gauge equivalence]
Two Maurer--Cartan elements are \emph{gauge equivalent} if there exists a gauge  $\lambda\in \F_1 A_0$ for which the flow of the vector field $-\d\lambda + \ad_\lambda$ relates them in finite time.
\end{definition}

In order to solve the associated differential equation, we consider the \emph{differential extension} 
\[\g^+\coloneq \left(A\oplus \k \delta, \d, \F, [\,,]\right)~, \quad \text{with}  \quad |\delta|=-1~,\  \d(\delta)=0~, \   
[\delta, x]= \d(x)~,\  \text{and} \ [\delta, \delta]=0~,\]
inside which we embed $\g$ as follows: $x\mapsto x$ when $|x|\neq -1$ and $x \mapsto \delta + x$ when $|x|=-1$~. 
In the differential extension, the Maurer-Cartan equation becomes the square-zero equation 
$[\delta +\alpha, \delta +\alpha]$ and the vector field becomes simply $\ad_\lambda(\delta+\alpha)$~. Therefore the 
solution to the associated differential equation starting at $\delta+\alpha$ is equal to 
\[\exp(t\,\ad_\lambda)(\delta+\alpha)=\left(\id +t\, \ad_\lambda + \tfrac12 t^2\, \ad_\lambda^2+\cdots  \right)(\delta+\alpha)~. \]
In the end, using the degree $0$ element $\lambda$, the Maurer-Cartan element $\alpha$ is gauge equivalent to 
\[\exp(\ad_\lambda)(\alpha)+\frac{\id-\exp(\ad_\lambda)}{\ad_{\lambda}}(\d\lambda)~.\]

The gauge equivalence actually comes from a group action as follows. When the complete dg Lie algebra is given by a complete dg unital associative algebra $\left(A, \d, \F, \star, 1\right)$ under skew-symmetrization of the product $[\,,]\coloneq \star - \star^{(12)}$~, 
the Maurer--Cartan equation is equal to 
\[ \d \alpha +\alpha \star \alpha =0~.\]
In this case, one can consider the topological complete group made up of \emph{the group-like elements}: 
\[\G\coloneq \left(1+\F_1A_0\,, \star\,, 1  \right)~.\]
Under the classical exponential 
$\exp(\lambda)\coloneq 1 + \lambda + \tfrac12 \lambda^{\star 2}+\tfrac16 \lambda^{\star 3} \cdots$
and logarithm $\ln(1+\lambda)\coloneq \lambda - \tfrac12 \lambda^{\star 2} +
\tfrac13 \lambda^{\star 3} + \cdots $ maps, this topological complete group is isomorphic to 
\[\Gamma\coloneq \left(\F_1A_0\,, \BCH\,, 0  \right)~,\]
where $\BCH(\lambda, \mu)\coloneq\ln\left(\exp(\lambda)\star \exp(\mu)\right)$~.
In this case and forgetting the differential, the \emph{gauge action} is given by the conjugation with the exponential:
\[\exp(\ad_\lambda)(\alpha)=\exp(\lambda)\star \alpha \star \exp(-\lambda)~.\]

\begin{theorem}[Baker--Campbell-Hausdorff, see \cite{BF12}]
Defined in the free complete unital associative algebra on two generators $x,y$, the \emph{Baker--Campbell--Hausdorff (BCH) formula} 
\[\BCH(x,y)\coloneq\ln\left(\exp(x) \exp(y)\right)\in \widehat{\Lie}(x,y)\]
actually lives in the free complete Lie algebra on $x$ and $y$. 
\end{theorem}

This celebrated theorem of Baker--Campbell-Hausdorff says that the BCH formula is actually made up of iterated brackets. Therefore it  shows that the gauge group is actually well-defined in any complete dg Lie algebra. 

\begin{definition}[Gauge group]
The \emph{gauge group} of a complete dg Lie algebra is the topological complete group defined by 
\[\Gamma\coloneq \left(\F_1A_0\,, \BCH\,, 0  \right)~.\]
\end{definition}

\begin{proposition}[Gauge group action]
In any complete dg Lie algebra, the gauge group acts continuously on Maurer--Cartan elements under the formula 
\[\exp(\ad_\lambda)(\alpha)+\frac{\id-\exp(\ad_\lambda)}{\ad_{\lambda}}(\d\lambda)~.\]
\end{proposition}

\begin{proof}[Sketch of proof]
This proof relies on the following characterisation of the BCH formula 
\[\exp\left(\ad_{\BCH(x,y)}\right) = \exp(\ad_x) \circ \exp(\ad_y) \]
established in \cite[Proposition~5.14]{Robert-NicoudVallette20}.
\end{proof}

One can know wonder how effective or algorithmic is the BCH formula. Some closed forms exist, the first and most famous one was given by Dynkin (1947)
\begin{multline*}
\BCH(x,y)=\sum_{n\geqslant 1} \frac{(-1)^{n-1}}{n}\sum_{p_i+q_i\geqslant 1 \atop i=\{1, \ldots, n-1\}}
\left(
\frac{\ad_x^{p_1} \circ \ad_y^{q_1}\circ \cdots \circ \ad_x^{p_{n-1}} \circ \ad_y^{q_n-1}(x) }{(1+\sum_{i=1}^{n-1} p_i+q_i)p_1!q_1! \ldots 
p_{n-1}!q_{n-1}!}
+\right.\\
\left.
\sum_{p_n\geqslant 0} \frac{\ad_x^{p_1} \circ \ad_y^{q_1}\circ \cdots \circ \ad_x^{p_{n-1}} \circ \ad_y^{q_n-1}\circ \ad_x^{p_n}(y)}{(p_n+1+\sum_{i=1}^{n-1} p_i+q_i)p_1!q_1! \ldots p_{n-1}!q_{n-1}!p_n!}\right)~.
\end{multline*}
However, the BCH formula and the formula the action on  Maurer--Cartan elements remain intricate and difficult to evaluate. 

\medskip

To summarize, like in the case of a Lie algebra given by a Lie group and related to it by an exponential map, when a complete dg Lie algebra comes from the skew-symmetrization of a complete dg unital associative algebra, the analytic exponential map defines a topological complete group isomorphism from the gauge group to the \emph{deformation gauge group}
\[\G= \left(1+\F_1A_0\,, \star\,, 1  \right)~.\]
Its action on Maurer--Cartan elements is much simpler since given by the conjugation, up to a constant term coming from the differential:
\[(1+\lambda)\star \alpha \star (1+\lambda)^{-1}- \d \lambda \star (1+\lambda)=
(1+\lambda)\star \alpha \star (1-\lambda+\lambda^{\star 2} -\lambda^{\star 3}+\cdots)- \d \lambda \star (1+\lambda)~.\]

\begin{application}
These effective formulae can be used to the complete dg unital associative algebra encoding morphisms of dg associative algebras, see \cite[Section~3]{DotsenkoShadrinVallette16}. 
\end{application}


\section{The stronger cases: pre-Lie and Lie-graph algebras}

There are two intermediate algebraic structures that sit between Lie algebras and associative algebras, pre-Lie algebras and Lie-graph algebras: the skew-symmetrization of their product produces a Lie bracket and any associative product satisfy their defining relations. 
In these two cases, one can also coin simpler and effective integration formulas for a group isomorphic to the gauge group and its action by "conjugation" on Maurer--Cartan elements. This  was developed in \cite{DotsenkoShadrinVallette16, DotsenkoShadrinVallette22, CamposVallette24} with in a view toward application to 
the deformation theories of morphisms of operads and properads respectively. 

\subsection{Pre-Lie algebras}

\begin{definition}[Pre-Lie algebra]
A \emph{pre-Lie algebra} is defined by a binary product whose associator is right symmetric: 
$$(x\star y)\star z - x\star (y\star z)=(-1)^{|y||z|}\big( (x\star z)\star y - x\star (z\star y)\big) \ . $$ 
\end{definition}

Associative products are particular examples of pre-Lie products and the skew-symmetrization of a pre-Lie product gives a Lie bracket. The Maurer--Cartan equation takes the simpler form 
\[\d \alpha + \alpha \star \alpha =0~.\]
We consider complete dg pre-Lie algebras $\left(A, \d, \F, \star, 1\right)$ which are \emph{left unital}, that is $1\star x = x$~. 
Since pre-Lie products fail to be associative in general, we need to consider the following product to define the relevant deformation gauge group on group-like elements.

\begin{definition}[Circle product]
For any  $y\in \F_1A_0$, 
the  \emph{circle product} is defined as the sum 
$${x \circledcirc (1+y) := \sum_{n\geqslant 0}  {\displaystyle \frac{1}{n!}}  \{x; \underbrace{y, \ldots, y}_{n}\}}$$
 of the \emph{symmetric braces}
$$\begin{array}{rcl}
\{ x; \}&:=&x \\
\{ x; y_1\}&:=&x\star y_1 \\
\{ x; y_1,y_2\}&:=& \{\{x; y_1\};  y_2\} - \{x; \{ y_1;  y_2\}\}= (x\star y_1)\star y_2 - x\star (y_1 \star y_2)\\
&&\\
\{ x; y_1,\ldots, y_n\}&:=&  \{ \{x; y_1,\ldots, y_{n-1}\}; y_n\}- \displaystyle \sum_{i=1}^{n-1} \{x; y_1, \ldots, y_{i-1}, 
\{y_i; y_n\}, y_{i+1}, \ldots, y_{n-1}\}   \ .\\
\end{array}
$$
\end{definition}

The last missing ingredient is the \emph{pre-Lie exponential map}
$$\exp(\lambda):=1 +\lambda + \frac{\lambda^{\star 2}}{2!} + \frac{\lambda^{\star 3}}{3!} +\cdots \ ,$$
which is defined by the right hand-side iteration of the pre-Lie product 
$${\lambda^{\star n}:=\underbrace{(\cdots((\lambda \star \lambda) \star \lambda)\cdots )\star \lambda}_{n\  \text{times}}}\ . $$
It admits an inverse map given by the "pre-Lie logarithm map" known as the \emph{Magnus expansion map} \cite{Manchon11}
$$\ln(1+\lambda)=\Omega(\lambda):=\lambda - \frac12 \lambda\star \lambda + \frac14 \lambda\star (\lambda \star \lambda)
+ \frac{1}{12}(\lambda \star \lambda)\star \lambda+\cdots \ .$$

\begin{theorem}[\cite{DotsenkoShadrinVallette16}]\label{thm:preLie}
Under the pre-Lie exponential map and the Magnus expansion map, the gauge group associated to 
any left-unital complete dg pre-Lie algebras is isomorphic to the \emph{deformation gauge group} made up of group-like elements equipped with the circle product:
$$
\Gamma\coloneq\left(\F_1A_0, \BCH, 0
\right)\cong  
\left(1+\F_1A_0, \circledcirc, 1\right) \eqcolon \G
\ .
$$
The latter acts on Maurer--Cartan elements under the following conjugation type formula 
$$\big((1+\lambda) \star \alpha\big) \circledcirc (1+\lambda)^{-1} - \d \lambda \circledcirc (1+\lambda)^{-1}  \ .$$
\end{theorem}

\begin{proof}[Sketch of proof]
The various elements of proof of this theorem can be found in \cite[Section~4]{DotsenkoShadrinVallette16}, except for the last term 
of the action formula. Instead of showing it here, we refer to the proof of the same statement (\cref{thm:Liegraph}) in the next and more general case (Lie-graph algebras), which includes the present one. 
\end{proof}

When the pre-Lie product is  associative, one  recovers the previous formulas.

\begin{application}
The Maurer--Cartan elements of the complete dg pre-Lie algebra 
$$  \left(
\Hom_{\Sy} \left({\overline{\P}}^{\ac}, \End_V\right),  \d, \star, {1}
\right)$$
associated to a Koszul operad $\P$ and a chain complex $V$ correspond to \emph{homotopy $\P$-algebra} structures, i.e. 
$\P_\infty=\Omega \P^{\ac}$-algebra structures,  on $V$.
In this case, the deformation gauge group is the group of \emph{$\infty$-isotopies} between such algebraic structures, that is the invertible $\infty$-morphisms whose first component is the identity of $V$, see \cite[Section~5]{DotsenkoShadrinVallette16}. 

\medskip 

These results were used in a crucial way in the proofs given in \cite{campos2019lie} of the following two fundamental theorems that were so far out of reach in representation theory and rational homotopy theory respectively: 
the universal enveloping algebra detects the isomorphism classes of nilpotent Lie algebras
and the rational homotopy type of a space is determined by its associative dg algebra of rational singular cochains. 
\end{application}

There is a little caveat: in the above displayed formula for the deformation gauge group action  appears the inverse can theoretically be computed by $(1+\lambda)^{-1}=\exp(-\Omega(\lambda))$, but we are lacking (so far) a closed formula for the Magnus expansion map. It turns out that the free pre-Lie algebra on a space $V$ is given by rooted trees \cite{ChapotonLivernet01} with vertices labelled by elements of $V$. Equivalently this means that rooted trees are in one-to-one correspondence with all the possible iterations of generic pre-Lie products. In this language,  the inverse of group-like elements for the circle product is given by the sum of rooted trees normalized by the cardinal of their automorphism group:
$$(1-\lambda)^{-1}=  \sum_{t\in \mathsf{RT}} 
\frac{1}{|\mathrm{Aut}\, t|}\, t(\lambda)\  .$$

\subsection{Lie-graph algebras}
In the associative and pre-Lie cases, there were clever ways to iterate the original binary product in order to produce all the required integration-deformation formulas. This method fails for general \emph{Lie-admissible products} that are products whose skew-symmetrization yields a Lie bracket. The  formula for the inverse mentioned just above suggests another approach: introduce a suitable algebraic structure with all its operations at once, not by generating operation(s) and their relation(s). This is precisely the main advantage of the notion of an operad introduced more than 50 years ago. 

\begin{definition}[Directed simple graph] 
	We call \emph{directed simple graph} a connected graph $\gra$, directed by a global flow from top to bottom (meaning that all edges are oriented downwards), with at least one vertex, and at most one edge between two vertices. 
	The number of vertices of a graph $\gra$ is denoted by $|\gra|$ and they are  labeled bijectively by $1,\ldots, |\gra|$. 
\end{definition}
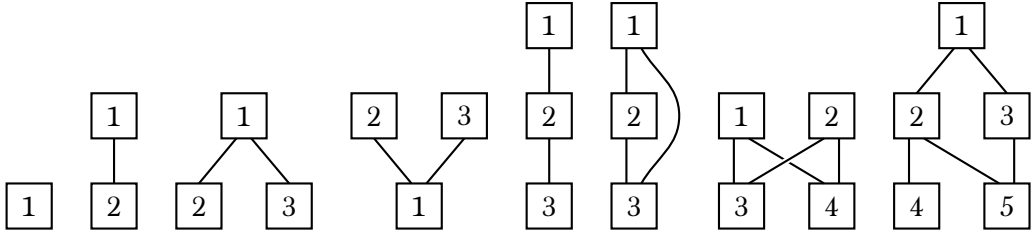
\begin{figure*}[h!]
	\begin{tikzpicture}[scale=0.6]
		\draw[thick]	(0,0)--(1,0)--(1,1)--(0,1)--cycle;
		\draw (0.5,0.5) node {{$1$}} ; 
	\end{tikzpicture} 
	\quad 
	\begin{tikzpicture}[scale=0.6]
		\draw[thick]	(0,0)--(1,0)--(1,1)--(0,1)--cycle;
		\draw[thick]	(0,2)--(1,2)--(1,3)--(0,3)--cycle;
		\draw[thick]	(0.5,1)--(0.5,2);
		\draw (0.5,0.5) node {{$2$}} ; 
		\draw (0.5,2.5) node {{$1$}} ; 
	\end{tikzpicture}
	\quad 
	\begin{tikzpicture}[scale=0.6]
		\draw[thick]	(0,0)--(1,0)--(1,1)--(0,1)--cycle;
		\draw[thick]	(2,0)--(3,0)--(3,1)--(2,1)--cycle;
		\draw[thick]	(1,2)--(2,2)--(2,3)--(1,3)--cycle;
		\draw[thick]	(0.5,1)--(1.33,2);
		\draw[thick]	(2.5,1)--(1.66,2);
		\draw (1.5,2.5) node {{$1$}} ; 
		\draw (0.5,0.5) node {{$2$}} ; 
		\draw (2.5,0.5) node {{$3$}} ; 	
	\end{tikzpicture}
	\quad 
	\begin{tikzpicture}[scale=0.6]
		\draw[thick]	(0,2)--(1,2)--(1,3)--(0,3)--cycle;
		\draw[thick]	(2,2)--(3,2)--(3,3)--(2,3)--cycle;
		\draw[thick]	(1,0)--(2,0)--(2,1)--(1,1)--cycle;
		\draw[thick]	(0.5,2)--(1.33,1);
		\draw[thick]	(2.5,2)--(1.66,1);
		\draw (1.5,0.5) node {{$1$}} ; 
		\draw (0.5,2.5) node {{$2$}} ; 
		\draw (2.5,2.5) node {{$3$}} ; 	
	\end{tikzpicture}
	\quad 
	\begin{tikzpicture}[scale=0.6]
		\draw[thick]	(0,0)--(1,0)--(1,1)--(0,1)--cycle;
		\draw[thick]	(0,2)--(1,2)--(1,3)--(0,3)--cycle;
		\draw[thick]	(0,4)--(1,4)--(1,5)--(0,5)--cycle;
		\draw[thick]	(0.5,1)--(0.5,2);
		\draw[thick]	(0.5,3)--(0.5,4);	
		\draw (0.5,0.5) node {{$3$}} ; 
		\draw (0.5,2.5) node {{$2$}} ; 
		\draw (0.5,4.5) node {{$1$}} ; 
	\end{tikzpicture}
	\quad 
	\begin{tikzpicture}[scale=0.6]
		\draw[thick]	(0,0)--(1,0)--(1,1)--(0,1)--cycle;
		\draw[thick]	(0,2)--(1,2)--(1,3)--(0,3)--cycle;
		\draw[thick]	(0,4)--(1,4)--(1,5)--(0,5)--cycle;
		\draw[thick]	(0.33,1)--(0.33,2);
		\draw[thick]	(0.33,3)--(0.33,4);
		\draw[thick]	(0.66,1) to[out=60,in=270] (1.5,2.5) to[out=90,in=300] (0.66,4);	
		\draw (0.5,0.5) node {{$3$}} ; 
		\draw (0.5,2.5) node {{$2$}} ; 
		\draw (0.5,4.5) node {{$1$}} ; 
	\end{tikzpicture}
	\quad
	\begin{tikzpicture}[scale=0.6]
		\draw[thick]	(0.66,2)--(2.33,1);	
		\draw[draw=white,double=black,double distance=2*\pgflinewidth,thick]	(0.66,1)--(2.33,2);		
		\draw[thick]	(0,0)--(1,0)--(1,1)--(0,1)--cycle;
		\draw[thick]	(2,0)--(3,0)--(3,1)--(2,1)--cycle;
		\draw[thick]	(0,2)--(1,2)--(1,3)--(0,3)--cycle;
		\draw[thick]	(2,2)--(3,2)--(3,3)--(2,3)--cycle;	
		\draw[thick]	(0.33,1)--(0.33,2);
		\draw[thick]	(2.66,1)--(2.66,2);
		\draw (0.5,2.5) node {{$1$}} ; 
		\draw (2.5,2.5) node {{$2$}} ; 	
		\draw (0.5,0.5) node {{$3$}} ; 
		\draw (2.5,0.5) node {{$4$}} ; 	
	\end{tikzpicture}
	\quad 
	\begin{tikzpicture}[scale=0.6]
		\draw[thick]	(0.66,2)--(2.33,1);	
		\draw[thick]	(0,0)--(1,0)--(1,1)--(0,1)--cycle;
		\draw[thick]	(2,0)--(3,0)--(3,1)--(2,1)--cycle;
		\draw[thick]	(0,2)--(1,2)--(1,3)--(0,3)--cycle;
		\draw[thick]	(2,2)--(3,2)--(3,3)--(2,3)--cycle;	
		\draw[thick]	(1,4)--(2,4)--(2,5)--(1,5)--cycle;
		\draw[thick]	(0.33,1)--(0.33,2);
		\draw[thick]	(2.66,1)--(2.66,2);
		\draw[thick]	(0.5,3)--(1.33,4);
		\draw[thick]	(2.5,3)--(1.66,4);	
		\draw (1.5,4.5) node {{$1$}} ; 
		\draw (0.5,2.5) node {{$2$}} ; 
		\draw (2.5,2.5) node {{$3$}} ; 	
		\draw (0.5,0.5) node {{$4$}} ; 
		\draw (2.5,0.5) node {{$5$}} ; 	
	\end{tikzpicture}
	\caption{The first  directed simple graphs.}
	\label{Fig:SimpleGraph}
\end{figure*}

The set of directed simple graphs is denoted by $\dcGra$ and we endow its linear span with the partial composition products $\gra_1 \circ_i \gra_2$ defined by the sum of directed simple graphs obtained by first inserting the graph $\gra_2$ at the vertex $i$ of the graph $\gra_1$, by shifting the indices of the vertices, and by replacing the edges between the vertices $j$ and the vertex $i$ in $\gra_1$ by all the possible ways to connect the vertices $j$ to the vertices of $\gra_2$, see \cref{Fig:PartCompo}. 

\begin{figure*}[h!]
	\begin{align*}
		\vcenter{\hbox{\begin{tikzpicture}[scale=0.6]
					\draw[thick]	(0,0)--(1,0)--(1,1)--(0,1)--cycle;
					\draw[thick]	(2,0)--(3,0)--(3,1)--(2,1)--cycle;
					\draw[thick]	(1,2)--(2,2)--(2,3)--(1,3)--cycle;
					\draw[thick]	(0.5,1)--(1.33,2);
					\draw[thick]	(2.5,1)--(1.66,2);
					\draw (1.5,2.5) node {{$2$}} ; 
					\draw (0.5,0.5) node {{$1$}} ; 
					\draw (2.5,0.5) node {{$3$}} ; 	
		\end{tikzpicture}}}
		& \ \circ_2\ 
		\vcenter{\hbox{
				\begin{tikzpicture}[scale=0.6]
					\draw[thick]	(0,0)--(1,0)--(1,1)--(0,1)--cycle;
					\draw[thick]	(0,2)--(1,2)--(1,3)--(0,3)--cycle;
					\draw[thick]	(0.5,1)--(0.5,2);
					\draw (0.5,0.5) node {{$2$}} ; 
					\draw (0.5,2.5) node {{$1$}} ; 
				\end{tikzpicture}
		}} = \ 
		\vcenter{\hbox{\begin{tikzpicture}[scale=0.6]
					\draw[thick] (0.66,1)--(1.33,2);
					\draw[thick] (2.33,1)--(1.66,2);
					\draw[thick]	(0,0)--(1,0)--(1,1)--(0,1)--cycle;
					\draw[thick]	(2,0)--(3,0)--(3,1)--(2,1)--cycle;
					\draw[thick]	(1,2)--(2,2)--(2,3)--(1,3)--cycle;
					\draw[thick]	(1,4)--(2,4)--(2,5)--(1,5)--cycle;	
					\draw[thick]	(1.5,3)--(1.5,4);	
					\draw (1.5,4.5) node {{$2$}} ; 
					\draw (1.5,2.5) node {{$3$}} ; 
					\draw (0.5,0.5) node {{$1$}} ; 
					\draw (2.5,0.5) node {{$4$}} ; 	
		\end{tikzpicture}}}
		\ + \ \vcenter{\hbox{\begin{tikzpicture}[scale=0.6]
					\draw[thick] (0.66,1)--(1.33,2);
					\draw[thick] (2.66,1) to (2.66,2.5) to[out=90,in=300]  (1.66,4);
					\draw[thick]	(0,0)--(1,0)--(1,1)--(0,1)--cycle;
					\draw[thick]	(2,0)--(3,0)--(3,1)--(2,1)--cycle;
					\draw[thick]	(1,2)--(2,2)--(2,3)--(1,3)--cycle;
					\draw[thick]	(1,4)--(2,4)--(2,5)--(1,5)--cycle;	
					\draw[thick]	(1.5,3)--(1.5,4);	
					\draw (1.5,4.5) node {{$2$}} ; 
					\draw (1.5,2.5) node {{$3$}} ; 
					\draw (0.5,0.5) node {{$1$}} ; 
					\draw (2.5,0.5) node {{$4$}} ; 	
		\end{tikzpicture}}}
		\ + \ \vcenter{\hbox{\begin{tikzpicture}[scale=0.6]
					\draw[thick] (0.33,1) to (0.33, 2.5) to[out=90,in=240] (1.33,4);
					\draw[thick] (2.33,1)--(1.66,2);
					\draw[thick]	(0,0)--(1,0)--(1,1)--(0,1)--cycle;
					\draw[thick]	(2,0)--(3,0)--(3,1)--(2,1)--cycle;
					\draw[thick]	(1,2)--(2,2)--(2,3)--(1,3)--cycle;
					\draw[thick]	(1,4)--(2,4)--(2,5)--(1,5)--cycle;	
					\draw[thick]	(1.5,3)--(1.5,4);	
					\draw (1.5,4.5) node {{$2$}} ; 
					\draw (1.5,2.5) node {{$3$}} ; 
					\draw (0.5,0.5) node {{$1$}} ; 
					\draw (2.5,0.5) node {{$4$}} ; 	
		\end{tikzpicture}}}
		\ + \ \vcenter{\hbox{\begin{tikzpicture}[scale=0.6]
					\draw[thick] (0.33,1) to (0.33, 2.5) to[out=90,in=240] (1.33,4);
					\draw[thick] (2.66,1) to (2.66,2.5) to[out=90,in=300]  (1.66,4);
					\draw[thick]	(0,0)--(1,0)--(1,1)--(0,1)--cycle;
					\draw[thick]	(2,0)--(3,0)--(3,1)--(2,1)--cycle;
					\draw[thick]	(1,2)--(2,2)--(2,3)--(1,3)--cycle;
					\draw[thick]	(1,4)--(2,4)--(2,5)--(1,5)--cycle;	
					\draw[thick]	(1.5,3)--(1.5,4);	
					\draw (1.5,4.5) node {{$2$}} ; 
					\draw (1.5,2.5) node {{$3$}} ; 
					\draw (0.5,0.5) node {{$1$}} ; 
					\draw (2.5,0.5) node {{$4$}} ; 	
		\end{tikzpicture}}} \\ &
		\ + \ \vcenter{\hbox{\begin{tikzpicture}[scale=0.6]
					\draw[thick] (0.66,1)--(1.33,2);
					\draw[thick] (0.33,1) to (0.33, 2.5) to[out=90,in=240] (1.33,4);
					\draw[thick] (2.33,1)--(1.66,2);
					\draw[thick]	(0,0)--(1,0)--(1,1)--(0,1)--cycle;
					\draw[thick]	(2,0)--(3,0)--(3,1)--(2,1)--cycle;
					\draw[thick]	(1,2)--(2,2)--(2,3)--(1,3)--cycle;
					\draw[thick]	(1,4)--(2,4)--(2,5)--(1,5)--cycle;	
					\draw[thick]	(1.5,3)--(1.5,4);	
					\draw (1.5,4.5) node {{$2$}} ; 
					\draw (1.5,2.5) node {{$3$}} ; 
					\draw (0.5,0.5) node {{$1$}} ; 
					\draw (2.5,0.5) node {{$4$}} ; 	
		\end{tikzpicture}}}
		\ + \ \vcenter{\hbox{\begin{tikzpicture}[scale=0.6]
					\draw[thick] (0.66,1)--(1.33,2);
					\draw[thick] (0.33,1) to (0.33, 2.5) to[out=90,in=240] (1.33,4);
					\draw[thick] (2.66,1) to (2.66,2.5) to[out=90,in=300]  (1.66,4);
					\draw[thick]	(0,0)--(1,0)--(1,1)--(0,1)--cycle;
					\draw[thick]	(2,0)--(3,0)--(3,1)--(2,1)--cycle;
					\draw[thick]	(1,2)--(2,2)--(2,3)--(1,3)--cycle;
					\draw[thick]	(1,4)--(2,4)--(2,5)--(1,5)--cycle;	
					\draw[thick]	(1.5,3)--(1.5,4);	
					\draw (1.5,4.5) node {{$2$}} ; 
					\draw (1.5,2.5) node {{$3$}} ; 
					\draw (0.5,0.5) node {{$1$}} ; 
					\draw (2.5,0.5) node {{$4$}} ; 	
		\end{tikzpicture}}}
		\ + \ \vcenter{\hbox{\begin{tikzpicture}[scale=0.6]
					\draw[thick] (0.66,1)--(1.33,2);
					\draw[thick] (2.33,1)--(1.66,2);
					\draw[thick] (2.66,1) to (2.66,2.5) to[out=90,in=300]  (1.66,4);
					\draw[thick]	(0,0)--(1,0)--(1,1)--(0,1)--cycle;
					\draw[thick]	(2,0)--(3,0)--(3,1)--(2,1)--cycle;
					\draw[thick]	(1,2)--(2,2)--(2,3)--(1,3)--cycle;
					\draw[thick]	(1,4)--(2,4)--(2,5)--(1,5)--cycle;	
					\draw[thick]	(1.5,3)--(1.5,4);	
					\draw (1.5,4.5) node {{$2$}} ; 
					\draw (1.5,2.5) node {{$3$}} ; 
					\draw (0.5,0.5) node {{$1$}} ; 
					\draw (2.5,0.5) node {{$4$}} ; 	
		\end{tikzpicture}}}
		\ + \ \vcenter{\hbox{\begin{tikzpicture}[scale=0.6]
					\draw[thick] (0.33,1) to (0.33, 2.5) to[out=90,in=240] (1.33,4);
					\draw[thick] (2.33,1)--(1.66,2);
					\draw[thick] (2.66,1) to (2.66,2.5) to[out=90,in=300]  (1.66,4);
					\draw[thick]	(0,0)--(1,0)--(1,1)--(0,1)--cycle;
					\draw[thick]	(2,0)--(3,0)--(3,1)--(2,1)--cycle;
					\draw[thick]	(1,2)--(2,2)--(2,3)--(1,3)--cycle;
					\draw[thick]	(1,4)--(2,4)--(2,5)--(1,5)--cycle;	
					\draw[thick]	(1.5,3)--(1.5,4);	
					\draw (1.5,4.5) node {{$2$}} ; 
					\draw (1.5,2.5) node {{$3$}} ; 
					\draw (0.5,0.5) node {{$1$}} ; 
					\draw (2.5,0.5) node {{$4$}} ; 	
		\end{tikzpicture}}}
		\ + \ \vcenter{\hbox{\begin{tikzpicture}[scale=0.6]
					\draw[thick] (0.66,1)--(1.33,2);
					\draw[thick] (0.33,1) to (0.33, 2.5) to[out=90,in=240] (1.33,4);
					\draw[thick] (2.33,1)--(1.66,2);
					\draw[thick] (2.66,1) to (2.66,2.5) to[out=90,in=300]  (1.66,4);
					\draw[thick]	(0,0)--(1,0)--(1,1)--(0,1)--cycle;
					\draw[thick]	(2,0)--(3,0)--(3,1)--(2,1)--cycle;
					\draw[thick]	(1,2)--(2,2)--(2,3)--(1,3)--cycle;
					\draw[thick]	(1,4)--(2,4)--(2,5)--(1,5)--cycle;	
					\draw[thick]	(1.5,3)--(1.5,4);	
					\draw (1.5,4.5) node {{$2$}} ; 
					\draw (1.5,2.5) node {{$3$}} ; 
					\draw (0.5,0.5) node {{$1$}} ; 
					\draw (2.5,0.5) node {{$4$}} ; 	
		\end{tikzpicture}}}
	\end{align*}
	\caption{Example of a partial composition in the operad $\Liegra$.}
	\label{Fig:PartCompo}
\end{figure*}
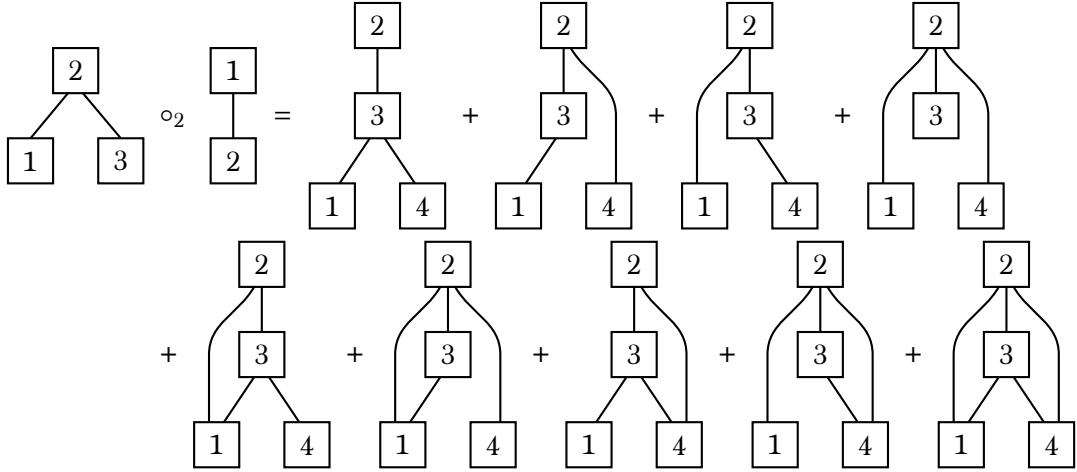

\begin{definition}[Lie-graph algebra]
	A \emph{Lie-graph algebra} is an algebra over the operad 
	\[\Liegra\coloneq (\kdcGra, \{\circ_i\}, \vcenter{\hbox{\begin{tikzpicture}[scale=0.3]
				\draw[thick]	(0,0)--(1,0)--(1,1)--(0,1)--cycle;
				\draw (0.5,0.5) node {{\tiny$1$}} ; 
	\end{tikzpicture}}} 
	)~.\]
\end{definition}

In any Lie-graph algebra, the binary product, denoted by $\star$, corresponding to the graph with two vertices is Lie-admissible, i.e. its skew-symmetrization produces a Lie bracket. So the Maurer--Cartan equation of dg Lie-graph algebras is 
\[\d\alpha +\alpha \star \alpha =0~.\]
The dimensions of the components of the Lie-graph operad are strictly greater than the ones of the Lie-admissible operad: this means that the iterations of the binary product $\star$ of a Lie-graph algebra fail to recover all the operations of this algebraic notion. 
Even worse, contrary to the operads $\Lie$, $\mathrm{Ass}$, $\Prelie$, $\mathrm{Lie}\textrm{-}\mathrm{adm}$ the operad $\Liegra$ is not finitely generated. 
This implies that we cannot bypass operad theory to even define $\Liegra$-algebras by means of a finite set of structural operations satisfying some relations.

\medskip

The projection of directed simple graphs onto rooted trees and the projection of rooted trees onto ladders define two surjective morphisms of operads
$$\Lie \rightarrowtail \Lieadm \to \Liegra \twoheadrightarrow \Prelie \twoheadrightarrow \mathrm{Ass}~.  $$
On the level of categories of algebras, this means that 
 pre-Lie algebras are  Lie-graph algebras whose operations vanish for graphs different from rooted trees
and associative algebras are pre-Lie algebras whose operations vanish for rooted trees  different from ladders. 

\begin{definition}[Graph circle product]
Given any complete dg Lie-graph algebra 
$(A,  \d, \F, \{\gra\}_{\gra\in \dcGra})$, we consider the set of formal group-like elements
\[\1+\F_1A_0\coloneq \{\1+x, \, x \in \F_1A_0\}~,\]
where $\1$ is a formal symbol. 
We equip it with the following \emph{graph circle product}:
\[
(\1+x)\circledcirc (\1+y)\coloneq \1 +\ \sum_{\gra \in {\ldcGra}_{\Sy}} {\textstyle \frac{1}{|\Aut(\gra)|}}\,  \gra(x,y)\ , 
\]
where $\gra(x,y)$ stands for the Lie-graph action of
	$2${-leveled directed simple graphs} 
 $\gra$ with bottom vertices labelled by $x$ and top vertices labelled by $y$.
\end{definition}

\begin{figure*}[h]
	\[(\1+x)\circledcirc (\1+y)\ =\ \1 \  + \
	\vcenter{\hbox{\begin{tikzpicture}[scale=0.4]
				\draw[thick]	(0,0)--(1,0)--(1,1)--(0,1)--cycle;
				\draw (0.5,0.5) node {{$x$}} ; 
	\end{tikzpicture}}}
	\ + \
	\vcenter{\hbox{\begin{tikzpicture}[scale=0.4]
				\draw[thick]	(0,0)--(1,0)--(1,1)--(0,1)--cycle;
				\draw (0.5,0.5) node {{$y$}} ; 
	\end{tikzpicture}}}
	\ + \
	\vcenter{\hbox{\begin{tikzpicture}[scale=0.4]
				\draw[thick]	(0,0)--(1,0)--(1,1)--(0,1)--cycle;
				\draw[thick]	(0,2)--(1,2)--(1,3)--(0,3)--cycle;
				\draw[thick]	(0.5,1)--(0.5,2);
				\draw (0.5,0.5) node {{$x$}} ; 
				\draw (0.5,2.5) node {{$y$}} ; 
	\end{tikzpicture}}}
	\ + \
	\scalebox{1.1}{$\frac{1}{2}$}\ 
	\vcenter{\hbox{\begin{tikzpicture}[scale=0.4]
				\draw[thick]	(0,0)--(1,0)--(1,1)--(0,1)--cycle;
				\draw[thick]	(2,0)--(3,0)--(3,1)--(2,1)--cycle;
				\draw[thick]	(1,2)--(2,2)--(2,3)--(1,3)--cycle;
				\draw[thick]	(0.5,1)--(1.33,2);
				\draw[thick]	(2.5,1)--(1.66,2);
				\draw (1.5,2.5) node {{$y$}} ; 
				\draw (0.5,0.5) node {{$x$}} ; 
				\draw (2.5,0.5) node {{$x$}} ; 	
	\end{tikzpicture}}}
	\ + \ 
	\scalebox{1.1}{$\frac{1}{2}$}\ 
	\vcenter{\hbox{\begin{tikzpicture}[scale=0.4]
				\draw[thick]	(0,2)--(1,2)--(1,3)--(0,3)--cycle;
				\draw[thick]	(2,2)--(3,2)--(3,3)--(2,3)--cycle;
				\draw[thick]	(1,0)--(2,0)--(2,1)--(1,1)--cycle;
				\draw[thick]	(0.5,2)--(1.33,1);
				\draw[thick]	(2.5,2)--(1.66,1);
				\draw (1.5,0.5) node {{$x$}} ; 
				\draw (0.5,2.5) node {{$y$}} ; 
				\draw (2.5,2.5) node {{$y$}} ; 	
	\end{tikzpicture}}}
	\ + \ 
	\scalebox{1.1}{$\frac{1}{4}$}\ 
	\vcenter{\hbox{\begin{tikzpicture}[scale=0.4]
				\draw[thick]	(0.66,2)--(2.33,1);	
				\draw[draw=white,double=black,double distance=2*\pgflinewidth,thick]	(0.66,1)--(2.33,2);		
				\draw[thick]	(0,0)--(1,0)--(1,1)--(0,1)--cycle;
				\draw[thick]	(2,0)--(3,0)--(3,1)--(2,1)--cycle;
				\draw[thick]	(0,2)--(1,2)--(1,3)--(0,3)--cycle;
				\draw[thick]	(2,2)--(3,2)--(3,3)--(2,3)--cycle;	
				\draw[thick]	(0.33,1)--(0.33,2);
				\draw[thick]	(2.66,1)--(2.66,2);
				\draw (0.5,2.5) node {{$y$}} ; 
				\draw (2.5,2.5) node {{$y$}} ; 	
				\draw (0.5,0.5) node {{$x$}} ; 
				\draw (2.5,0.5) node {{$x$}} ; 	
	\end{tikzpicture}}}
	\ + \ \cdots\]
	\caption{The first terms of the graph circle product $\cc$\ .}
	\label{Fig:GrpProd}
\end{figure*}

\begin{definition}[Graph exponential and logarithm maps]
	For any complete dg Lie-graph algebra, the \emph{graph exponential map} is defined by 
	\[ \begin{array}{lccc}
		\exp\ : & \F_1 A_0 & \to &\1 +\F_1 A_0\\
		&\lambda&\mapsto& \displaystyle\1 + \sum_{\gra \in \dcGra_{\Sy}} \frac{\ell_\gra}{|\gra|!}\,  \gra(\lambda)\ ,
	\end{array} \]
where $\ell_\gra$ is equal to the number of ways to place the vertices of the graph $\gra$ on $|\gra|$ levels. 	
It admits an inverse map called the  the \emph{graph logarithm map}:
$\ln\ :  \1+\F_1 A_0 \to  \F_1 A_0$~. 
\begin{figure*}[h]
	\[\exp(\lambda)\ =\ \1 \  + \
	\vcenter{\hbox{\begin{tikzpicture}[scale=0.4]
				\draw[thick]	(0,0)--(1,0)--(1,1)--(0,1)--cycle;
				\draw (0.5,0.5) node {{$\lambda$}} ; 
	\end{tikzpicture}}}
	\ + \
	\scalebox{1.1}{$\frac{1}{2}$}\ 
	\vcenter{\hbox{\begin{tikzpicture}[scale=0.4]
				\draw[thick]	(0,0)--(1,0)--(1,1)--(0,1)--cycle;
				\draw[thick]	(0,2)--(1,2)--(1,3)--(0,3)--cycle;
				\draw[thick]	(0.5,1)--(0.5,2);
				\draw (0.5,0.5) node {{$\lambda$}} ; 
				\draw (0.5,2.5) node {{$\lambda$}} ; 
	\end{tikzpicture}}}
	\ + \
	\scalebox{1.1}{$\frac{1}{6}$}\ 
	\vcenter{\hbox{\begin{tikzpicture}[scale=0.4]
				\draw[thick]	(0,0)--(1,0)--(1,1)--(0,1)--cycle;
				\draw[thick]	(2,0)--(3,0)--(3,1)--(2,1)--cycle;
				\draw[thick]	(1,2)--(2,2)--(2,3)--(1,3)--cycle;
				\draw[thick]	(0.5,1)--(1.33,2);
				\draw[thick]	(2.5,1)--(1.66,2);
				\draw (1.5,2.5) node {{$\lambda$}} ; 
				\draw (0.5,0.5) node {{$\lambda$}} ; 
				\draw (2.5,0.5) node {{$\lambda$}} ; 	
	\end{tikzpicture}}}
	\ + \ 
	\scalebox{1.1}{$\frac{1}{6}$}\ 
	\vcenter{\hbox{\begin{tikzpicture}[scale=0.4]
				\draw[thick]	(0,2)--(1,2)--(1,3)--(0,3)--cycle;
				\draw[thick]	(2,2)--(3,2)--(3,3)--(2,3)--cycle;
				\draw[thick]	(1,0)--(2,0)--(2,1)--(1,1)--cycle;
				\draw[thick]	(0.5,2)--(1.33,1);
				\draw[thick]	(2.5,2)--(1.66,1);
				\draw (1.5,0.5) node {{$\lambda$}} ; 
				\draw (0.5,2.5) node {{$\lambda$}} ; 
				\draw (2.5,2.5) node {{$\lambda$}} ; 	
	\end{tikzpicture}}}
	\ + \ 
	\scalebox{1.1}{$\frac{1}{6}$}\ 
	\vcenter{\hbox{\begin{tikzpicture}[scale=0.4]
				\draw[thick]	(0,0)--(1,0)--(1,1)--(0,1)--cycle;
				\draw[thick]	(0,2)--(1,2)--(1,3)--(0,3)--cycle;
				\draw[thick]	(0,4)--(1,4)--(1,5)--(0,5)--cycle;
				\draw[thick]	(0.5,1)--(0.5,2);
				\draw[thick]	(0.5,3)--(0.5,4);	
				\draw (0.5,0.5) node {{$\lambda$}} ; 
				\draw (0.5,2.5) node {{$\lambda$}} ; 
				\draw (0.5,4.5) node {{$\lambda$}} ; 
	\end{tikzpicture}}}
	\ + \ 
	\scalebox{1.1}{$\frac{1}{6}$}\ 
	\vcenter{\hbox{\begin{tikzpicture}[scale=0.4]
				\draw[thick]	(0,0)--(1,0)--(1,1)--(0,1)--cycle;
				\draw[thick]	(0,2)--(1,2)--(1,3)--(0,3)--cycle;
				\draw[thick]	(0,4)--(1,4)--(1,5)--(0,5)--cycle;
				\draw[thick]	(0.33,1)--(0.33,2);
				\draw[thick]	(0.33,3)--(0.33,4);
				\draw[thick]	(0.66,1) to[out=60,in=270] (1.5,2.5) to[out=90,in=300] (0.66,4);	
				\draw (0.5,0.5) node {{$\lambda$}} ; 
				\draw (0.5,2.5) node {{$\lambda$}} ; 
				\draw (0.5,4.5) node {{$\lambda$}} ; 
	\end{tikzpicture}}}
	\ + \
	\scalebox{1.1}{$\frac{1}{8}$}\ 
	\vcenter{\hbox{\begin{tikzpicture}[scale=0.4]
				\draw[thick]	(0,0)--(1,0)--(1,1)--(0,1)--cycle;
				\draw[thick]	(2,0)--(3,0)--(3,1)--(2,1)--cycle;
				\draw[thick]	(1,2)--(2,2)--(2,3)--(1,3)--cycle;
				\draw[thick]	(0,-2)--(1,-2)--(1,-1)--(0,-1)--cycle;	
				\draw[thick]	(0.5,1)--(1.33,2);
				\draw[thick]	(2.5,1)--(1.66,2);
				\draw[thick]	(0.5,0)--(0.5,-1);	
				\draw (0.5,-1.5) node {{$\lambda$}} ; 
				\draw (1.5,2.5) node {{$\lambda$}} ; 
				\draw (0.5,0.5) node {{$\lambda$}} ; 
				\draw (2.5,0.5) node {{$\lambda$}} ; 	
	\end{tikzpicture}}}
	\ + \ 
	\scalebox{1.1}{$\frac{1}{24}$}\ 
	\vcenter{\hbox{\begin{tikzpicture}[scale=0.4]
				\draw[thick]	(0.66,2)--(2.33,1);	
				\draw[draw=white,double=black,double distance=2*\pgflinewidth,thick]	(0.66,1)--(2.33,2);		
				\draw[thick]	(0,0)--(1,0)--(1,1)--(0,1)--cycle;
				\draw[thick]	(2,0)--(3,0)--(3,1)--(2,1)--cycle;
				\draw[thick]	(0,2)--(1,2)--(1,3)--(0,3)--cycle;
				\draw[thick]	(2,2)--(3,2)--(3,3)--(2,3)--cycle;	
				\draw[thick]	(0.33,1)--(0.33,2);
				\draw[thick]	(2.66,1)--(2.66,2);
				\draw (0.5,2.5) node {{$\lambda$}} ; 
				\draw (2.5,2.5) node {{$\lambda$}} ; 	
				\draw (0.5,0.5) node {{$\lambda$}} ; 
				\draw (2.5,0.5) node {{$\lambda$}} ; 	
	\end{tikzpicture}}}
	\ + \ \cdots\]
	\[\ln(\1+\lambda)\ =
	\vcenter{\hbox{\begin{tikzpicture}[scale=0.4]
				\draw[thick]	(0,0)--(1,0)--(1,1)--(0,1)--cycle;
				\draw (0.5,0.5) node {{$\lambda$}} ; 
	\end{tikzpicture}}}
	\ - \
	\scalebox{1.1}{$\frac{1}{2}$}\ 
	\vcenter{\hbox{\begin{tikzpicture}[scale=0.4]
				\draw[thick]	(0,0)--(1,0)--(1,1)--(0,1)--cycle;
				\draw[thick]	(0,2)--(1,2)--(1,3)--(0,3)--cycle;
				\draw[thick]	(0.5,1)--(0.5,2);
				\draw (0.5,0.5) node {{$\lambda$}} ; 
				\draw (0.5,2.5) node {{$\lambda$}} ; 
	\end{tikzpicture}}}
	\ + \
	\scalebox{1.1}{$\frac{1}{12}$}\ 
	\vcenter{\hbox{\begin{tikzpicture}[scale=0.4]
				\draw[thick]	(0,0)--(1,0)--(1,1)--(0,1)--cycle;
				\draw[thick]	(2,0)--(3,0)--(3,1)--(2,1)--cycle;
				\draw[thick]	(1,2)--(2,2)--(2,3)--(1,3)--cycle;
				\draw[thick]	(0.5,1)--(1.33,2);
				\draw[thick]	(2.5,1)--(1.66,2);
				\draw (1.5,2.5) node {{$\lambda$}} ; 
				\draw (0.5,0.5) node {{$\lambda$}} ; 
				\draw (2.5,0.5) node {{$\lambda$}} ; 	
	\end{tikzpicture}}}
	\ + \ 
	\scalebox{1.1}{$\frac{1}{12}$}\ 
	\vcenter{\hbox{\begin{tikzpicture}[scale=0.4]
				\draw[thick]	(0,2)--(1,2)--(1,3)--(0,3)--cycle;
				\draw[thick]	(2,2)--(3,2)--(3,3)--(2,3)--cycle;
				\draw[thick]	(1,0)--(2,0)--(2,1)--(1,1)--cycle;
				\draw[thick]	(0.5,2)--(1.33,1);
				\draw[thick]	(2.5,2)--(1.66,1);
				\draw (1.5,0.5) node {{$\lambda$}} ; 
				\draw (0.5,2.5) node {{$\lambda$}} ; 
				\draw (2.5,2.5) node {{$\lambda$}} ; 	
	\end{tikzpicture}}}
	\ + \ 
	\scalebox{1.1}{$\frac{1}{3}$}\ 
	\vcenter{\hbox{\begin{tikzpicture}[scale=0.4]
				\draw[thick]	(0,0)--(1,0)--(1,1)--(0,1)--cycle;
				\draw[thick]	(0,2)--(1,2)--(1,3)--(0,3)--cycle;
				\draw[thick]	(0,4)--(1,4)--(1,5)--(0,5)--cycle;
				\draw[thick]	(0.5,1)--(0.5,2);
				\draw[thick]	(0.5,3)--(0.5,4);	
				\draw (0.5,0.5) node {{$\lambda$}} ; 
				\draw (0.5,2.5) node {{$\lambda$}} ; 
				\draw (0.5,4.5) node {{$\lambda$}} ; 
	\end{tikzpicture}}}
	\ + \ 
	\scalebox{1.1}{$\frac{1}{3}$}\ 
	\vcenter{\hbox{\begin{tikzpicture}[scale=0.4]
				\draw[thick]	(0,0)--(1,0)--(1,1)--(0,1)--cycle;
				\draw[thick]	(0,2)--(1,2)--(1,3)--(0,3)--cycle;
				\draw[thick]	(0,4)--(1,4)--(1,5)--(0,5)--cycle;
				\draw[thick]	(0.33,1)--(0.33,2);
				\draw[thick]	(0.33,3)--(0.33,4);
				\draw[thick]	(0.66,1) to[out=60,in=270] (1.5,2.5) to[out=90,in=300] (0.66,4);	
				\draw (0.5,0.5) node {{$\lambda$}} ; 
				\draw (0.5,2.5) node {{$\lambda$}} ; 
				\draw (0.5,4.5) node {{$\lambda$}} ; 
	\end{tikzpicture}}}
	\ + \ \cdots\]
	\caption{The first terms of the graph exponential and logarithm maps.}
\end{figure*}
\end{definition}

\begin{theorem}[\cite{CamposVallette24}]\label{thm:Liegraph}
Under the graph exponential and logarithm maps, the gauge group associated to 
any complete dg Lie-graph algebra is isomorphic to the \emph{deformation gauge group} made up of formal group-like elements equipped with the graph circle product:
$$
\Gamma\coloneq\left(\F_1A_0, \BCH, 0
\right)\cong  
\left(\1+\F_1A_0, \circledcirc, 1\right) \eqcolon \G
\ .
$$
\end{theorem}

\begin{proof}[Sketch of proof]
The detailed proof can be found in \cite[Section~2]{CamposVallette24}. The first idea amounts to work in the free complete Lie-graph algebra on two generators with rational coefficients in order to produce universal formulas. The second main idea amounts to apply the classical Lie theory: viewed with real coefficients and truncated modulo any layer of the complete filtration, one gets a finite dimensional real Lie algebra and its associated Lie group. Finally, the graph exponential map is showed to be the classical Lie theoretical exponential map between them. 
\end{proof}

It remains to make effective the formula for the action of this deformation gauge group on Maurer--Cartan elements. 

\begin{definition}[Bowtie element]
	For any elements $x,y \in\F_1A_0$ and $\alpha\in A$ in a complete Lie-graph algebra, we consider the associated \emph{bowtie element}
	of $A$~:
	\[\pap{(\1+x)}{\alpha}{(\1+y)}\coloneq \sum_{\gra \in \btGra} {\textstyle \frac{1}{|\Aut(\gra)|}} \, \gra(x,\alpha,y)\ ,\]
where the sum runs over the actions of  
3-leveled directed simple graphs with only one vertex on the middle level labelled by $\alpha$ and (possibly empty) top and bottom levels (labelled respectively by $y$ and $x$). 
	\begin{figure*}[h]
		\begin{align*}
			\pap{(\1+x)}{\alpha}{(\1+y)}\ = \ &
			\vcenter{\hbox{\begin{tikzpicture}[scale=0.4]
						\draw[thick]	(0,0)--(1,0)--(1,1)--(0,1)--cycle;
						\draw (0.5,0.5) node {{$\alpha$}} ; 
			\end{tikzpicture}}}
			\ + \
			\vcenter{\hbox{\begin{tikzpicture}[scale=0.4]
						\draw[thick]	(0,0)--(1,0)--(1,1)--(0,1)--cycle;
						\draw[thick]	(0,2)--(1,2)--(1,3)--(0,3)--cycle;
						\draw[thick]	(0.5,1)--(0.5,2);
						\draw (0.5,0.5) node {{$x$}} ; 
						\draw (0.5,2.5) node {{$\alpha$}} ; 
			\end{tikzpicture}}}
			\ + \
			\vcenter{\hbox{\begin{tikzpicture}[scale=0.4]
						\draw[thick]	(0,0)--(1,0)--(1,1)--(0,1)--cycle;
						\draw[thick]	(0,2)--(1,2)--(1,3)--(0,3)--cycle;
						\draw[thick]	(0.5,1)--(0.5,2);
						\draw (0.5,0.5) node {{$\alpha$}} ; 
						\draw (0.5,2.5) node {{$y$}} ; 
			\end{tikzpicture}}}
			\ + \
			\scalebox{1.1}{$\frac{1}{2}$}\ 
			\vcenter{\hbox{\begin{tikzpicture}[scale=0.4]
						\draw[thick]	(0,0)--(1,0)--(1,1)--(0,1)--cycle;
						\draw[thick]	(2,0)--(3,0)--(3,1)--(2,1)--cycle;
						\draw[thick]	(1,2)--(2,2)--(2,3)--(1,3)--cycle;
						\draw[thick]	(0.5,1)--(1.33,2);
						\draw[thick]	(2.5,1)--(1.66,2);
						\draw (1.5,2.5) node {{$\alpha$}} ; 
						\draw (0.5,0.5) node {{$x$}} ; 
						\draw (2.5,0.5) node {{$x$}} ; 	
			\end{tikzpicture}}}
			\ + \ 
			\scalebox{1.1}{$\frac{1}{2}$}\ 
			\vcenter{\hbox{\begin{tikzpicture}[scale=0.4]
						\draw[thick]	(0,2)--(1,2)--(1,3)--(0,3)--cycle;
						\draw[thick]	(2,2)--(3,2)--(3,3)--(2,3)--cycle;
						\draw[thick]	(1,0)--(2,0)--(2,1)--(1,1)--cycle;
						\draw[thick]	(0.5,2)--(1.33,1);
						\draw[thick]	(2.5,2)--(1.66,1);
						\draw (1.5,0.5) node {{$\alpha$}} ; 
						\draw (0.5,2.5) node {{$y$}} ; 
						\draw (2.5,2.5) node {{$y$}} ; 	
			\end{tikzpicture}}}
			\ + \ \\&
			\vcenter{\hbox{\begin{tikzpicture}[scale=0.4]
						\draw[thick]	(0,0)--(1,0)--(1,1)--(0,1)--cycle;
						\draw[thick]	(0,2)--(1,2)--(1,3)--(0,3)--cycle;
						\draw[thick]	(0,4)--(1,4)--(1,5)--(0,5)--cycle;
						\draw[thick]	(0.5,1)--(0.5,2);
						\draw[thick]	(0.5,3)--(0.5,4);	
						\draw (0.5,0.5) node {{$x$}} ; 
						\draw (0.5,2.5) node {{$\alpha$}} ; 
						\draw (0.5,4.5) node {{$y$}} ; 
			\end{tikzpicture}}}
			\ + \  
			\vcenter{\hbox{\begin{tikzpicture}[scale=0.4]
						\draw[thick]	(0,0)--(1,0)--(1,1)--(0,1)--cycle;
						\draw[thick]	(0,2)--(1,2)--(1,3)--(0,3)--cycle;
						\draw[thick]	(0,4)--(1,4)--(1,5)--(0,5)--cycle;
						\draw[thick]	(0.33,3)--(0.33,4);
						\draw[thick]	(0.66,1) to[out=60,in=270] (1.5,2.5) to[out=90,in=300] (0.66,4);	
						\draw (0.5,0.5) node {{$x$}} ; 
						\draw (0.5,2.5) node {{$\alpha$}} ; 
						\draw (0.5,4.5) node {{$y$}} ; 
			\end{tikzpicture}}}
			\ + \
			\vcenter{\hbox{\begin{tikzpicture}[scale=0.4]
						\draw[thick]	(0,0)--(1,0)--(1,1)--(0,1)--cycle;
						\draw[thick]	(0,2)--(1,2)--(1,3)--(0,3)--cycle;
						\draw[thick]	(0,4)--(1,4)--(1,5)--(0,5)--cycle;
						\draw[thick]	(0.33,1)--(0.33,2);
						\draw[thick]	(0.66,1) to[out=60,in=270] (1.5,2.5) to[out=90,in=300] (0.66,4);	
						\draw (0.5,0.5) node {{$x$}} ; 
						\draw (0.5,2.5) node {{$\alpha$}} ; 
						\draw (0.5,4.5) node {{$y$}} ; 
			\end{tikzpicture}}}
			\ + \  
			\vcenter{\hbox{\begin{tikzpicture}[scale=0.4]
						\draw[thick]	(0,0)--(1,0)--(1,1)--(0,1)--cycle;
						\draw[thick]	(0,2)--(1,2)--(1,3)--(0,3)--cycle;
						\draw[thick]	(0,4)--(1,4)--(1,5)--(0,5)--cycle;
						\draw[thick]	(0.33,1)--(0.33,2);
						\draw[thick]	(0.33,3)--(0.33,4);
						\draw[thick]	(0.66,1) to[out=60,in=270] (1.5,2.5) to[out=90,in=300] (0.66,4);	
						\draw (0.5,0.5) node {{$x$}} ; 
						\draw (0.5,2.5) node {{$\alpha$}} ; 
						\draw (0.5,4.5) node {{$y$}} ; 
			\end{tikzpicture}}}
			\ + \ 
			\scalebox{1.1}{$\frac{1}{2}$}\ 
			\vcenter{\hbox{\begin{tikzpicture}[scale=0.4]
						\draw[thick]	(0,0)--(1,0)--(1,1)--(0,1)--cycle;
						\draw[thick]	(2,0)--(3,0)--(3,1)--(2,1)--cycle;
						\draw[thick]	(1,2)--(2,2)--(2,3)--(1,3)--cycle;
						\draw[thick]	(1,4)--(2,4)--(2,5)--(1,5)--cycle;	
						\draw[thick]	(0.5,1)--(1.33,2);
						\draw[thick]	(2.5,1)--(1.66,2);
						\draw[thick]	(1.5,3)--(1.5,4);		
						\draw (1.5,2.5) node {{$\alpha$}} ; 
						\draw (0.5,0.5) node {{$x$}} ; 
						\draw (2.5,0.5) node {{$x$}} ; 	
						\draw (1.5,4.5) node {{$y$}} ; 	
			\end{tikzpicture}}}
			\ + \ 
			\vcenter{\hbox{\begin{tikzpicture}[scale=0.4]
						\draw[thick] (0.66,1)--(1.33,2);
						\draw[thick] (2.33,1)--(1.66,2);
						\draw[thick] (2.66,1) to (2.66,2.5) to[out=90,in=300]  (1.66,4);
						\draw[thick]	(0,0)--(1,0)--(1,1)--(0,1)--cycle;
						\draw[thick]	(2,0)--(3,0)--(3,1)--(2,1)--cycle;
						\draw[thick]	(1,2)--(2,2)--(2,3)--(1,3)--cycle;
						\draw[thick]	(1,4)--(2,4)--(2,5)--(1,5)--cycle;	
						\draw[thick]	(1.5,3)--(1.5,4);	
						\draw (1.5,4.5) node {{$y$}} ; 
						\draw (1.5,2.5) node {{$\alpha$}} ; 
						\draw (0.5,0.5) node {{$x$}} ; 
						\draw (2.5,0.5) node {{$x$}} ; 	
			\end{tikzpicture}}}
			\ + \ \cdots
		\end{align*}
		\caption{The first terms of the bowtie element 
			$(\1+x)\stackrel{\alpha}{\Join} (\1+y)$~.
		}
		\label{Fig:Pap}
	\end{figure*}
\end{definition}

\begin{theorem}[\cite{CamposVallette24}]
The deformation gauge group acts on Maurer--Cartan elements under the following formula 
\[\pap{(\1 +\lambda)}{\alpha}{(\1 +\lambda)^{-1}}-\left(\1+\lambda; \d\lambda \right)\cc (\1+\lambda)^{-1}~,\]
where the last term is given by the sum of the actions of 2-leveled directed simple graphs with one bottom vertex labelled by $\d \lambda$ and all the other bottom vertices labelled by $\lambda$.
\end{theorem}

\begin{proof}[Sketch of proof]
The detailed proof can be found in \cite[Section~2.3]{CamposVallette24}. It relies on classical  Lie theoretical methods again and the introduction of the notion of the differential extension of a dg Lie-graph algebras, which is more subtle than in the Lie case where one-dimensional extensions were enough. 
\end{proof}

The inverse of formal group-like element is explicitly  given by 
\[(\1+\lambda)^{-1}=\1+\displaystyle \sum_{\gra \in \dcGra_{\Sy}} \frac{(-1)^{|\gra|}}{|\mathrm{Aut}(\gra)|} \, \gra(\lambda)~. \]

For Lie-graph algebras whose operations vanish for graphs different from rooted trees, one recovers all the formulas given previously for pre-Lie algebras.

\begin{application}
The Maurer--Cartan elements of the complete dg Lie-admissible algebra 
$$  \left(
\Hom_{\Sy} \left({\overline{\P}}^{\ac}, \End_V\right),  \d, \star, {1}
\right)$$
associated to a Koszul properad $\P$ and a chain complex $V$ correspond to \emph{homotopy $\P$-gebra} structures, i.e. 
$\P_\infty=\Omega \P^{\ac}$-gebra structures, on $V$. 
The operad of ``all the natural operations'' acting on this deformation complex of morphisms of properads is the operad $\Liegra$ introduced above. 
Like in the operadic case, the deformation gauge group is the group of \emph{$\infty$-isotopies} between homotopy $\P$-gebras, see \cite[Section~1]{CV25II}. 
This includes the case of pre-Calabi--Yau algebras and homotopy versions of (involutive) Lie bialgebras, Frobenius bialgebras, Airy structures and double Poisson bialgebras, for instance. 

\medskip 

The  Lie-graph integration-deformation context plays a key role in the definition of the \emph{Kaledin--Emprin classes} which faithfully detect the formality of algebraic structures encoded by groupoid colored properads \cite{Emprin24}. 
Considering the exponential side of the theory, namely the the group of $\infty$-isotopies, allows C. Emprin 
to successfully address the intricate formality property  in prime characteristic. 
\end{application}


\section{The higher cases: $\mathrm{L}_\infty$-algebras and absolute curved $\mathrm{EL}_\infty$-algebras}

So far we have been considering only the Maurer--Cartan elements (degree $-1$) and the gauges (degree $0$). One may wonder which role could play the higher degree elements? Degree $1$ elements can actually be interpreted as comparisons between gauges, degree $2$ elements as comparisons between these comparisons, etc. Considering all these higher gauges will pave the way to the effective integration theory for the homotopy meaningful notions associated to dg Lie algebras: $\Li$-algebra in characteristic $0$ \cite{Robert-NicoudVallette20} and absolute curved $\mathrm{EL}_\infty$-algebra in prime characteristic \cite{Roca23}. 

\subsection{$\mathrm{L}_\infty$-algebras} 

\begin{definition}[Shifted $\Li$-algebras]
A \emph{shifted $\Li$-algebra} $\g=\left(A, \d, \{\ell_m\}_{m \geqslant 2}\right)$ is a chain complex equipped with  symmetric operations $\ell_n \colon A^{\odot m}\coloneq \left(A^{\otimes m}\right)_{\Sy_n} \to A$ of degree $-1$ satisfying 
	\[
	\partial\left(\ell_m\right)+\sum_{\substack{p+q=m\\ 2\leqslant p,q \leqslant m}}
	\sum_{\sigma\in \mathrm{Sh}_{p,q}^{-1}}
	(\ell_{p+1}\circ_{1} \ell_q)^{\sigma}=0\ ,
	\]
	for any $m\geqslant2$~, where $ \mathrm{Sh}_{p,q}^{-1}$ denotes the set of the inverses of $(p,q)$-shuffles and where $\partial\left(\ell_m\right)\coloneq \d \circ \ell_m +\ell_m\circ  \d_{A^{\odot m}}$~.  
\end{definition}

A dg Lie algebra structure on a chain complex $(A,\d)$ is a shifted $\Li$-algebra structure on the suspension $sA$ such that the higher bracket $\ell_m$ vanish for $m\geqslant 3$. 
We denote by $\sLi$ the dg operad encoding shifted $\Li$-algebras. 
A \emph{complete $s\Li$-algebra} is a shifted $\Li$-algebra in the symmetric monoidal category of complete chain complexes.
	We denote by $\sLialg$ the associated category made up of 
	strict morphisms, i.e. single maps preserving the structural operations. 
The notion of an $\Li$-algebra contains more elaborate signs. 
Contrary to one might think at first sight, the theory of (shifted) $\Li$-algebras is \emph{simpler} than that of Lie algebras. For instance, the free (shifted)  $\Li$-algebra $\sLi(V)$ on a chain complex $V$ is easy to make explicit, on the opposite to the free Lie algebra: it is simply given by rooted trees with leaves labelled by elements of $V$. 
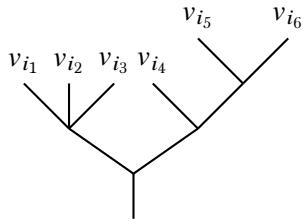
\[
		\begin{tikzpicture}
		\def\scale{0.6};
		\pgfmathsetmacro{\diagcm}{sqrt(2)};
		
		\def\xangle{35};
		\pgfmathsetmacro{\xcm}{1/sin(\xangle)};
		
		\coordinate (r) at (0,0);
		\coordinate (v11) at ($(r) + (0,\scale*1)$);
		\coordinate (v21) at ($(v11) + (180-\xangle:\scale*\xcm)$);
		\coordinate (v22) at ($(v11) + (\xangle:\scale*\xcm)$);
		\coordinate (v31) at ($(v22) + (45:\scale*\diagcm)$);
		\coordinate (l1) at ($(v21) + (135:\scale*\diagcm)$);
		\coordinate (l2) at ($(v21) + (0,\scale*1)$);
		\coordinate (l3) at ($(v21) + (45:\scale*\diagcm)$);
		\coordinate (l4) at ($(v22) + (135:\scale*\diagcm)$);
		\coordinate (l5) at ($(v31) + (135:\scale*\diagcm)$);
		\coordinate (l6) at ($(v31) + (45:\scale*\diagcm)$);
		
		\draw[thick] (r) to (v11);
		\draw[thick] (v11) to (v21);
		\draw[thick] (v11) to (v22);
		\draw[thick] (v21) to (l1);
		\draw[thick] (v21) to (l2);
		\draw[thick] (v21) to (l3);
		\draw[thick] (v22) to (l4);
		\draw[thick] (v22) to (v31);
		\draw[thick] (v31) to (l5);
		\draw[thick] (v31) to (l6);
		
		\node[above] at (l1) {$v_{i_1}$};
		\node[above] at (l2) {$v_{i_2}$};
		\node[above] at (l3) {$v_{i_3}$};
		\node[above] at (l4) {$v_{i_4}$};
		\node[above] at (l5) {$v_{i_5}$};
		\node[above] at (l6) {$v_{i_6}$};
		\end{tikzpicture}
\]

\begin{definition}[Maurer--Cartan element]
A \emph{Maurer--Cartan element} of a complete $s\Li$-algebra $\g$ is a degree $0$ element $\alpha\in \F_1 A_0$ satisfying the \emph{Maurer--Cartan equation}: 
\[
\d \alpha+\sum_{m\geqslant 2} {\textstyle \frac{1}{m!}}\ell_m(\alpha, \ldots, \alpha)=0 \ .
\]
\end{definition}

We denote by $\MC(\g)$ the set of Maurer-Cartan elements. In order to get cleaner formulas, one often introduced the convention $\ell_1\coloneq \d$~. 
The degree $1$ elements $\lambda \in \F_1A_1$ are called \emph{the gauges}; they canonically give rise to a vector fields 
\[\alpha \mapsto \sum_{m\geqslant 1} {\textstyle \frac{1}{(m-1)!}}\ell_m(\alpha, \cdots, \alpha, \lambda) \in \mathrm{T}_\alpha  \mathrm{MC}(\g)\ .\]

\begin{definition}[Gauge equivalence]
Two Maurer--Cartan elements are \emph{gauge equivalent} if there exists a gauge for which the flow of the associated vector field relates them in finite time.
\end{definition}

This gauge equivalence relation admits the following explicit formula.

\begin{proposition}[{\cite[Proposition~5.9]{Getzler09} and \cite[Proposition~1.10]{Robert-NicoudVallette20}}]
Using the gauge $\lambda \in \F_1A_1$, any Maurer--Cartan element $\alpha$ is gauge equivalent (at time $1$) to 
\[\sum_{\tau\in \mathsf{PRT}}{\textstyle \frac{1}{C(\tau)}}\tau^\lambda(\alpha)~,\]
where the sum runs over planar rooted trees, 
where the decomposition of a planar rooted tree $\tau = \mathrm{c}_m\circ(\tau_1,\ldots,\tau_m)$ as the grafting of $m$ planar rooted trees $\tau_1,\ldots,\tau_m$ along the leaves of 
the root corolla $\mathrm{c}_m$ defines recursively the coefficients $C(\tau)\in \mathbb{N}$ by
\[
C(\, |\, ) \coloneq 1\ ,\qquad C\left(\mathrm{c}_m\right) \coloneq m!\ ,\qquad C(\tau) \coloneq m! |\tau|\prod_{i=1}^m C(\tau_i)\ ,
\]
and the elements $\tau^\lambda(\alpha)\in A_0$ by
\[
|^\lambda(\alpha)\coloneq \alpha\ ,\qquad \mathrm{c}_m^\lambda(\alpha)\coloneq\ell_{m+1}(\alpha, \ldots, \alpha, \lambda)\ ,\qquad \tau^\lambda(\alpha) \coloneq \ell_{m+1}\left(\tau_1^\lambda(\alpha),\ldots, \tau_m^\lambda(\alpha), \lambda \right).
\]
\end{proposition}

\begin{example}
For the planar rooted tree 
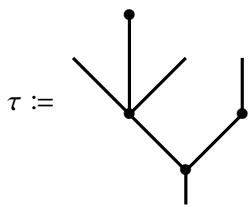
\[
\tau\coloneq \vcenter{\hbox{
	\begin{tikzpicture}
		\def\scale{0.75}
		\pgfmathsetmacro{\diagcm}{sqrt(2)};
		
		\coordinate (r) at (0,0);
		
		\coordinate (v11) at ($(r) + (0,\scale*0.6)$);
		
		\coordinate (v21) at ($(v11) + (135:\scale*\diagcm)$);
		\coordinate (v22) at ($(v11) + (45:\scale*\diagcm)$);
		
		\coordinate (v31) at ($(v21) + (0,\scale*1.75)$);
		
		\coordinate (l1) at ($(v21) + (135:\scale*\diagcm)$);
		\coordinate (l2) at ($(v21) + (45:\scale*\diagcm)$);
		\coordinate (l3) at ($(v22) + (0,\scale*1)$);
		
		\node at (v11) {$\bullet$};
		\node at (v21) {$\bullet$};
		\node at (v22) {$\bullet$};
		\node at (v31) {$\bullet$};
		
		\draw[very thick] (r) to (v11);
		\draw[very thick] (v11) to (v21);
		\draw[very thick] (v11) to (v22);
		\draw[very thick] (v21) to (v31);
		\draw[very thick] (v21) to (l1);
		\draw[very thick] (v21) to (l2);
		\draw[very thick] (v22) to (l3);
	\end{tikzpicture}}}
\]
the coefficient $C(\tau)$ is equal to $96$ and the element $\tau^\lambda(\alpha)$ is equal to 
\[
\tau^\lambda(\alpha) = \ell_3(\ell_4(\alpha, \d \lambda, \alpha, \lambda), \ell_2(\alpha, \lambda), \lambda)\ ,
\] 
which corresponds graphically to 
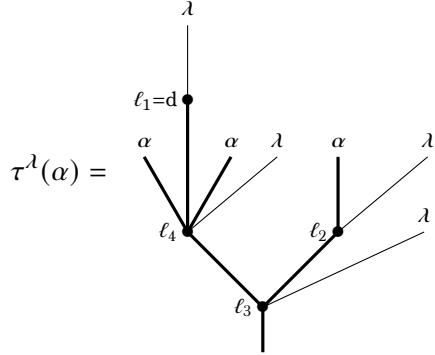
\[\tau^\lambda(\alpha)= \vcenter{\hbox{
\begin{tikzpicture}
	\def\scale{1}
	\pgfmathsetmacro{\diagcm}{sqrt(2)};
	\pgfmathsetmacro{\xcm}{1/sin(60)};
	
	\def\angle{40}
	\pgfmathsetmacro{\ycm}{1/sin(\angle)};
	
	\def\anglebis{25}
	\pgfmathsetmacro{\zcm}{1/sin(\anglebis)};
	
	\coordinate (r) at (0,0);
	
	\coordinate (v11) at ($(r) + (0,\scale*0.6)$);
	
	\coordinate (v21) at ($(v11) + (135:\scale*\diagcm)$);
	\coordinate (v22) at ($(v11) + (45:\scale*\diagcm)$);
	
	\coordinate (v31) at ($(v21) + (0,\scale*1.75)$);
	
	\coordinate (l1) at ($(v21) + (120:\scale*\xcm)$);
	\coordinate (l2) at ($(v21) + (60:\scale*\xcm)$);
	\coordinate (l3) at ($(v22) + (0,\scale*1)$);
	
	\coordinate (lam1) at ($(v31) + (0,\scale*1)$);
	\coordinate (lam2) at ($(v21) + (\angle:\scale*\ycm)$);
	\coordinate (lam3) at ($(v22) + (\angle:\scale*\ycm)$);
	\coordinate (lam4) at ($(v11) + (\anglebis:\scale*\zcm)$);
	
	\node at (v11) {$\bullet$};
	\node at (v21) {$\bullet$};
	\node at (v22) {$\bullet$};
	\node at (v31) {$\bullet$};
	
	\draw[very thick] (r) to (v11);
	\draw[very thick] (v11) to (v21);
	\draw[very thick] (v11) to (v22);
	\draw[very thick] (v21) to (v31);
	\draw[very thick] (v21) to (l1);
	\draw[very thick] (v21) to (l2);
	\draw[very thick] (v22) to (l3);
	
	\draw[thin] (v31) to (lam1);
	\draw[thin] (v21) to (lam2);
	\draw[thin] (v22) to (lam3);
	\draw[thin] (v11) to (lam4);
	
	\node[above] at (l1) {$\scriptstyle{\alpha}$};
	\node[above] at (l2) {$\scriptstyle{\alpha}$};
	\node[above] at (l3) {$\scriptstyle{\alpha}$};
	
	\node[above] at (lam1) {$\scriptstyle{\lambda}$};
	\node[above] at (lam2) {$\scriptstyle{\lambda}$};
	\node[above] at (lam3) {$\scriptstyle{\lambda}$};
	\node[above] at (lam4) {$\scriptstyle{\lambda}$};
	
	\node[left] at (v11) {$\scriptstyle{\ell_3}$};
	\node[left] at (v21) {$\scriptstyle{\ell_4}$};
	\node[left] at (v22) {$\scriptstyle{\ell_2}$};
	\node[left] at (v31) {$\scriptstyle{\ell_1 = \d}$};
\end{tikzpicture}}}
\]
\end{example}

There is no hope to write this equivalence relation as a gauge group action: since the Jacobi relation does not hold any more, the BCH formula fails to be associative. One solution is to look for a much wider context, i.e. include the higher gauges into the picture. In the Lie case, the data of the Maurer--Cartan elements and the action of the gauges form the Deligne groupoid. Heuristically speaking, it is made up of points, some of which related by paths. Higher up, one expects to have homotopies between paths, and homotopies between homotopies, etc. In the end, this should form some kind of topological space, a combinatorial model of which is provided by the notion of a \emph{Kan complex}, also known as an \emph{$\infty$-groupoid}, that is a simplicial set such that every horn can be filled. 

\medskip

Recall that the data of a pair of adjoint functors 
\[
\hbox{
	\begin{tikzpicture}
	\def\upshift{0.185}
	\def\downshift{0.15}
	\pgfmathsetmacro{\midshift}{0.005}
	
	\node[left] (x) at (0, 0) {$\mathrm{L}\ \ :\ \ \sSe$};
	\node[right] (y) at (2, 0) {$\sLialg\ \ :\ \ \mathrm{R}$};
	
	\draw[-{To[left]}] ($(x.east) + (0.1, \upshift)$) -- ($(y.west) + (-0.1, \upshift)$);
	\draw[-{To[left]}] ($(y.west) + (-0.1, -\downshift)$) -- ($(x.east) + (0.1, -\downshift)$);
	
	\node at ($(x.east)!0.5!(y.west) + (0, \midshift)$) {\scalebox{0.8}{$\perp$}};
	\end{tikzpicture}}
\]
is equivalent to the data of a cosimplicial object in the right-hand side category \cite{Kan58bis}, 
corresponding to the restriction of the functor $\mathrm{L}$ 
to the image of the Yoneda embedding. 
This provides us with an efficient way to construct a functor from complete  $s\Li$-algebras to simplicial sets. 
To this extend, one can consider Dupont's simplicial contraction \cite{Dupont76}:
\[
\hbox{
	\begin{tikzpicture}
	\def\upshift{0.075}
	\def\downshift{0.075}
	\pgfmathsetmacro{\midshift}{0.005}
	
	\node[left] (x) at (0, 0) {$\Omega_\bullet$};
	\node[right=1.5 cm of x] (y) {$\mathrm{C}_\bullet$};
	
	\draw[-{To}] ($(x.east) + (0.1, \upshift)$) -- node[above]{\mbox{\tiny{$p_\bullet$}}} ($(y.west) + (-0.1, \upshift)$);
	\draw[-{To}] ($(y.west) + (-0.1, -\downshift)$) -- node[below]{\mbox{\tiny{$i_\bullet$}}} ($(x.east) + (0.1, -\downshift)$);
	
	\draw[->] ($(x.south west) + (0, 0.1)$) to [out=-160,in=160,looseness=5] node[left]{\mbox{\tiny{$h_\bullet$}}} ($(x.north west) - (0, 0.1)$);
	
	\end{tikzpicture}}
\]
between the piece-wise polynomial differential forms 
and the cellular cochains 
of the standard geometric simplices. 
The former being a simplicial commutative algebras, one can apply the homotopy transfer theorem to endow the latter one with a kind of simplicial homotopy commutative algebra structure. Since it is degree-wise finite dimensional, one can consider its linear dual: the cellular chain complexes of the standard geometric simplices $\mathrm{C}^\bullet$ carry a canonical cosimplicial homotopy commutative coalgebra structure. By Koszul duality, its image under a suitable cobar functor $\hatCobar_\pi$ produces the required cosimplicial complete  $s\Li$-algebra. 

\begin{definition}[Universal Maurer--Cartan algebra]
	The cosimplicial complete  $\sLi$-algebra
	\[
	\mc^\bullet\coloneq \hatCobar_\pi \mathrm{C}^\bullet\cong \left(\widehat{\sLi}\big(C(\De{\bullet})\big), \d\right)\ .
	\]
	is called the \emph{universal Maurer--Cartan algebra}\index{universal Maurer--Cartan algebra}.
\end{definition}

We denote by $a_I$, for $\emptyset\neq I\subseteq [n]$, the cells of the geometric $n$-simplex. For $n=0$, one gets the quasi-free complete $\sLi$-algebra on one Maurer--Cartan element:
\[\mc^0\cong \left(\widehat{\sLi}(a_0), \d\right)\,, \quad \text{with} \quad  \d a_0 =-\sum_{m\geqslant 2} {\textstyle \frac{1}{m!}}\ell_m(a_0, \ldots, a_0)\ .
\]
For $n=1$, one gets the quasi-free complete $\sLi$-algebra on two Maurer--Cartan elements and a gauge relating them:
\[\mc^1\cong \left(\widehat{\sLi}(a_0, a_1, a_{01}), \d\right)~.\]

\begin{definition}[Integration functor]\label{def:SimpRep}
The  \emph{integration functor} of complete $\sLi$-algebras is defined by 
\begin{align*}
\mathrm{R}\ :\ \sLialg{}&\ \longrightarrow\ \sSe\\
\g{}&\ \longmapsto\ \Hom_{\,\sLialg}\left(\mc^\bullet, \g\right).
\end{align*}
\end{definition}

A direct corollary of the explicit form for $\mc^0$ and $\mc^1$ given above is 
\[
\mathrm{R}(\g)_0\cong \MC(\g)\quad \text{and} \quad \pi_0(\mathrm{R}(\g))\cong \mathcal{MC}(\g)~,
\]
where the latter stands for the moduli space of Maurer--Cartan elements up to gauge equivalences. 
By definition, the integration functor $\R$ comes equipped with the following left adjoint functor 
\[\mathrm{L}\coloneq \mathrm{Lan}_\mathrm{Y} \mc^\bullet\ : \ \sSe \to \sLialg~,
\] 
where $\mathrm{Lan}_\mathrm{Y}$ stands of the left Kan extension along the Yoneda embedding. 

\begin{Historical}
The definition of a first integration functor can be traced through D. Sullivan's foundational article on rational homotopy theory \cite{Sullivan77};  it was extensively studied by V. Hinich in \cite{Hinich97bis}. It is explicitly given by 
\[
\MC_\bullet(\g)\coloneq\MC\left(\g \, \widehat{\otimes}\, \Omega_\bullet\right),
\]
where the tensor product of a complete $\sLi$-algebra with a simplicial commutative algebra carries a canonical simplicial complete $\sLi$-algebra structure. The $0$-simplicies coincide with the Maurer--Cartan elements $\MC_0(\g)\cong\MC(\g)$ but the set of $1$-simplicies $\MC_1(\g)$ is way too big to contain only the gauges (but it has the same homotopy type). E. Getzler introduced in \cite{Getzler09} the following gauge condition 
\[ \MC_\bullet(\g)\cap \ker h_\bullet\cong \R(\g)~, \]
which fixes this issue. Bypassing this intrinsic definition, Getzler's functor was proved to be 
represented by the aforementioned universal Maurer--Cartan algebra in \cite{Robert-NicoudVallette20}.
\end{Historical}

One can also endow $\g \, \widehat{\otimes}\, \mathrm{C}_\bullet$ with a complete $\sLi$-algebra structure such that 
\[\R(\g)_n \cong \MC\left(\g \, \widehat{\otimes}\, C\big(\De{n}\big)^\vee\right) \ni x=\sum_{\emptyset\neq I\subseteq [n]} x_I \otimes a_I^\vee~, \]
with the degree equal to $|x_I|=|I|-1$~. So its elements can be represented on the geometric $n$-simplex by 
	
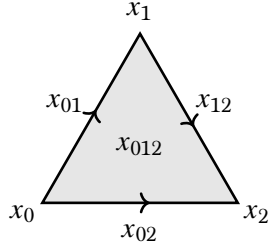
\[
	\begin{tikzpicture}
	
		\coordinate (v0) at (210:1.5);
		\coordinate (v1) at (90:1.5);
		\coordinate (v2) at (-30:1.5);
		
		\fill[black, opacity=0.1] (v0)--(v1)--(v2)--cycle;
		\draw[line width=1] (v0)--(v1)--(v2)--cycle;
		
		\begin{scope}[decoration={
			markings,
			mark=at position 0.55 with {\arrow{>}}},
			line width=1
		]
			
			\path[postaction={decorate}] (v0) -- (v1);
			\path[postaction={decorate}] (v1) -- (v2);
			\path[postaction={decorate}] (v0) -- (v2);
		
		\end{scope}
		
		
		\node at ($(v0) + (30:-0.3)$) {$x_0$};
		\node at ($(v1) + (0,0.3)$) {$x_1$};
		\node at ($(v2) + (-30:0.3)$) {$x_2$};
		
		\node at ($(v0)!0.5!(v1) + (-30:-0.4)$) {$x_{01}$};
		\node at ($(v1)!0.5!(v2) + (30:0.4)$) {$x_{12}$};
		\node at ($(v0)!0.5!(v2) + (0,-0.4)$) {$x_{02}$};
		
		\node at (0,0) {$x_{012}$};
		
	\end{tikzpicture}
	\]
which provides us with an  efficiently way to work out the properties of the representation functor.

\medskip

It remains to prove that the integration functor produces an $\infty$-groupoid. In the present case, there exists a neat description of the set of horn fillers. 

\begin{theorem}[\cite{Robert-NicoudVallette20}]\label{thm:hornfillers}
There is a canonical natural bijection 
\begin{equation*}
\left\{\vcenter{\hbox{
	\begin{tikzpicture}
		\node (a) at (0, 0) {$\Ho{k}{n}$};
		\node (b) at (2, 0) {$\R(\g)$};
		\node (c) at (0, -1.5) {$\De{n}$};
		
		\draw[->] (a) to (b);
		\draw[right hook->] (a) to node[below=-0.05cm, sloped]{$\sim$} (c);
		\draw[->, dashed] (c) to (b);
	\end{tikzpicture}
}}\right\} 
\cong \g_n\ .
\end{equation*}
\end{theorem}

\begin{proof}[Sketch of proof]
The proof given in \cite[Section~5.1]{Robert-NicoudVallette20} relies on the following use of the left-adjoint functor, which satisfies 
	\[
	\mc^n=\mathrm{L}\big(\De{n}\big)\cong \mathrm{L}\big(\Ho{k}{d}\big)\,\widehat{\sqcup}\, \widehat{\sLi}(u, du)\ ,
	\]
for $n\geqslant 2$~, where $u$ is a generator of degree $d$ and where the differential is given by $d(u)=du$~.
This implies that there exists a natural bijection 	
\begin{equation}\label{eq1}\tag{$\ast$}
	\Hom_\sSe\left(\De{n}, \R(\g)\right)\cong  \Hom_\sSe\left(\Ho{k}{n}, \R(\g)\right)\times  \g_n~.
\end{equation}
\end{proof}

This proves that the set of horn fillers is non-empty since $0\in \g_n$~, but, even more interesting, it carries a canonical element that can be considered from the beginning in the definition of an $\infty$-groupoid. This changes radically the nature of $\R(\g)$ making it leave the realm of homotopy theory (defined by a property) 
and casting it into the world of algebra (defined by some structure). 

\begin{corollary} The image $\R(\g)$ of any complete $\sLi$-algebra $\g$ under the integration functor is canonically an \emph{algebraic} $\infty$-groupoid. 
\end{corollary}

\begin{definition}[Higher Baker--Campbell--Hausdorff products]
	We call  \emph{higher Baker--Camp\-bell--Haus\-dorff products}\index{higher Baker--Campbell--Hausdorff products} the maps
	\begin{align*}
	\Hom_{\sSe}\left({\Ho{k}{n}}, \R(\g)\right) &\ \longrightarrow\ \g_{n-1}\\
	x &\ \longmapsto\ \Gamma^n_k(x)
	\end{align*}
	defined by the evaluation at the cell $a_{0, \ldots, \hat{k}, \ldots, n}$ of $\De{n}$ of the $n$-simplex inverse to 
	$(x;0)$ under the bijection \eqref{eq1}.
\end{definition}

Let $\left(A, \F, [\,,]\right)$ be a complete Lie algebra; its suspension $sA$ forms a complete $\sLi$-algebra concentrated in degree $1$. Therefore, it is made up of the trivial Maurer-Cartan element and trivial higher gauges, for any $n\geqslant 2$; only its set of gauges $\F_1 A_0$ might not be trivial. In this case, the data of a horn ${\Ho{1}{2}} \to \R(\g)$ is equivalent to the data of two elements $x$ and $y$ in $\F_1A_0$.

\begin{proposition}[{\cite[Proposition~5.2.36]{Bandiera14}}]
For any complete Lie algebra, the canonical  horn filler
\[\Gamma^2_1(x, y)=\mathrm{BCH}(x,y)\]
	is equal to the Baker--Campbell--Hausdorff formula.
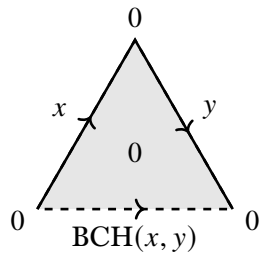
\[
\begin{tikzpicture}

\coordinate (v0) at (210:1.5);
\coordinate (v1) at (90:1.5);
\coordinate (v2) at (-30:1.5);

\fill[black, opacity=0.1] (v0)--(v1)--(v2)--cycle;
\draw[line width=1, dashed] (v0)--(v1)--(v2)--cycle;
\draw[line width=1] (v0)--(v1)--(v2);

\begin{scope}[decoration={
	markings,
	mark=at position 0.55 with {\arrow{>}}},
line width=1
]

\path[postaction={decorate}] (v0) -- (v1);
\path[postaction={decorate}] (v1) -- (v2);
\path[postaction={decorate}] (v0) -- (v2);

\end{scope}


\node at ($(v0) + (30:-0.3)$) {$0$};
\node at ($(v1) + (0,0.3)$) {$0$};
\node at ($(v2) + (-30:0.3)$) {$0$};

\node at ($(v0)!0.5!(v1) + (-30:-0.4)$) {$x$};
\node at ($(v1)!0.5!(v2) + (30:0.4)$) {$y$};
\node at ($(v0)!0.5!(v2) + (0,-0.4)$) {$\BCH(x,y)$};

\node at (0,0) {$0$};

\end{tikzpicture}
\]
\end{proposition}

So the very first example of horn fillers provides us with the most universal formula in Lie theory, which is in fact obtained via (simplicial) homotopy theory. We can pursue this study further since the \emph{constructive} proof of \cref{thm:hornfillers} allows us to give \emph{explicit formulas} for all the horn fillers. To do so, we consider the set $\PaPRT$ of \emph{planarly partitioned rooted trees}, which are rooted trees with the additional data of a \emph{partition} into sub-trees called \emph{blocks} satisfying
\begin{itemize}
	\item[$\diamond$] each block contains at least one vertex,
	\item[$\diamond$] each block is either a tree with vertices of arity at least $2$ or the single $0$-corolla, and
	\item[$\diamond$] the tree obtained by stripping the partitioned tree of all of its leaves and contracting the blocks to vertices is planar but the rooted sub-trees inside each block have no planar structure.
\end{itemize}
\begin{figure*}[h!]
	\[
\vcenter{\hbox{
		\begin{tikzpicture}
			\def\scale{0.75};
			\pgfmathsetmacro{\diagcm}{sqrt(2)};
			
			\def\xangle{35};
			\pgfmathsetmacro{\xcm}{1/sin(\xangle)};
			
			\coordinate (r) at (0,0);
			\coordinate (v11) at ($(r) + (0,\scale*1)$);
			\coordinate (v21) at ($(v11) + (180-\xangle:\scale*\xcm)$);
			\coordinate (v22) at ($(v11) + (\xangle:\scale*\xcm)$);
			\coordinate (v31) at ($(v22) + (45:\scale*\diagcm)$);
			\coordinate (l1) at ($(v21) + (135:\scale*\diagcm)$);
			\coordinate (l2) at ($(v21) + (0,\scale*2)$);
			\coordinate (l3) at ($(v21) + (45:\scale*\diagcm)$);
			\coordinate (l4) at ($(v22) + (135:\scale*\diagcm)$);
			\coordinate (l5) at ($(v31) + (135:\scale*\diagcm)$);
			\coordinate (l6) at ($(v31) + (45:\scale*\diagcm)$);
			\coordinate (l7) at ($(v11) + (0,\scale*1)$);
			
			\draw[thick] (r) to (v11);
			\draw[thick] (v11) to (v21);
			\draw[thick] (v11) to (v22);
			\draw[thick] (v21) to (l1);
			\draw[thick] (v21) to (l2);
			\draw[thick] (v21) to (l3);
			\draw[thick] (v22) to (l4);
			\draw[thick] (v22) to (v31);
			\draw[thick] (v31) to (l5);
			\draw[thick] (v31) to (l6);
			\draw[thick] (v11) to (l7);
			
			\node[above] at (l1) {$\scriptstyle1$};
			\node at (l2) {$\bullet$};
			\node[above] at (l3) {$\scriptstyle2$};
			\node[above] at (l4) {$\scriptstyle4$};
			\node[above] at (l5) {$\scriptstyle5$};
			\node[above] at (l6) {$\scriptstyle6$};
			\node[above] at (l7) {$\scriptstyle3$};
			
			\draw (v11) circle[radius=\scale*0.5];
			\draw (v21) circle[radius=\scale*0.5];
			\draw (l2) circle[radius=\scale*0.5];
			\draw ($(v22)!0.5!(v31)$) ellipse[x radius=\scale*1.3, y radius=\scale*0.6, rotate=45];
		\end{tikzpicture}
	}}
	\]
	\caption{Example of a planarly partitioned rooted tree.}
	\label{Fig:PaPRT}
\end{figure*}
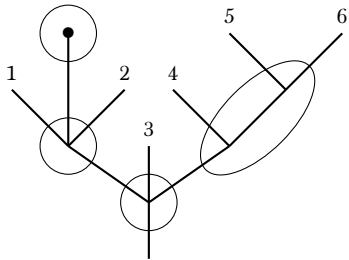
We consider the set $\mathsf{Lab}^{[n], k}(\tau)$ of maps 
\[
\chi \ : \ \mathrm{Leaves}(\tau) \longrightarrow\left\{I\varsubsetneq[n]\,\mid\, I\neq\emptyset,\widehat{k}\right\}~, 
\]
where $\widehat{k}\coloneq I\backslash \{k\}$, 
which amont to labelling the leaves of $\tau$~. 
For any horn $x \colon {\Ho{k}{n}} \to \R(\g)$ and any element 
$z\in \g$~,  we denote by $\ell_\tau(x_{\chi(1)}, \ldots, x_{\chi(p)};  z)$ the element of $\g$ obtained by forgetting the partition, replacing the leaves $l$ by the corresponding elements $x_{\chi(l)}\in\g$ and the vertices of arity $0$ by $z$~, and then applying the structure operations of the algebra at the vertices of the tree with the corresponding arity. In the example of the planarly partitioned rooted tree $\tau$ given in \cref{Fig:PaPRT}, we get 
	\[\ell_\tau(x_{\chi(1)}, \ldots, x_{\chi(6)};  z)=
	\ell_3\left(
	\ell_3\left(x_{\chi(1)}, z, x_{\chi(2)}\right)
	x_{\chi(3)}, 
	\ell_2\left(x_{\chi(4)}, \ell_2\left(x_{\chi(5)}, x_{\chi(6)}\right)\right)
	\right)~.
	\]

\begin{theorem}[{\cite[Proposition~5.10]{Robert-NicoudVallette20}}]\label{prop:ExplicitBCHForm}
	The higher BCH products are given by 
	\begin{equation*}
	\Gamma^n_k(x) = 
	\sum_{\substack{\tau\in\PaPRT\\
			\chi\in\mathsf{Lab}^{[n], k}(\tau)}}\ 
	\prod_{\substack{\beta\text{ block of } \tau \\ \lambda^{\beta(\chi)}_{[n]}\neq 0}} 
	\frac{(-1)^k}{\lambda^{\beta(\chi)}_{[n]}[\beta]!}\, 
	\ell_\tau\left(x_{\chi(1)}, \ldots, x_{\chi(p)};   \sum_{l\neq k}(-1)^{k+l+1} x_{\widehat{l}}\right)~, 
	\end{equation*}
where  $[\beta]$ denotes the arity of the block $\beta$ and where the coefficients $\lambda^{\beta(\chi)}_{[n]}$ are the structure constants  of the simplicial homotopy commutative algebras structure on $\mathrm{C}_\bullet$~, see \cite[Equation~(3)]{Robert-NicoudVallette20}.
\end{theorem}

\begin{proof}[Sketch of proof]
This formula is a direct corollary of \cref{thm:hornfillers} and the solution to the fixed-point equation associated to  analytic functions given in \cite[Proposition~A.5]{Robert-NicoudVallette20}. 
\end{proof}

\begin{application}
These higher BCH products allow us to give in \cite[Section~6.1]{Robert-NicoudVallette20} a short proof of a functorial version of Berglund's Hurewicz type theorem \cite{ber15}: for any complete $\sLi$-algebra $\g$ and any Maurer--Cartan element $\alpha \in \MC(\g)$~, there is a functorial group isomorphism
	\[
	\pi_n(\R(\g), \alpha)\cong {H}_n(\g^\alpha)\ , \quad \text{for}\  \ n\geqslant 1\ ,
	\]
	where $\g^\alpha$ stands for the complete $\sLi$-algebra twisted by $\alpha$ \cite[Chapter~4]{DotsenkoShadrinVallette22} and 
	where the group structure on the right-hand side is given by the Baker--Campbell--Hausdorff formula  for $n=1$  and by the sum  for $n\geqslant 2$~. 
	
\medskip 

These higher BCH products also allow us to prove the homotopy invariance of the integration functor: it sends filtered $\infty_\pi$-quasi-isomorphisms (a refined notion of $\infty$-quasi-isomorphisms) of complete $\sLi$-algebras to weak equivalences of simplicial sets \cite[Section~6.2]{Robert-NicoudVallette20}. This is an advanced generalization of the Goldman--Millson theorem \cite{GoldmanMillson88} saying that quasi-isomorphic dg Lie algebra encode equivalent deformation problems. 

\medskip

This higher Lie theory provides us with well behaved model of the rational homotopy type of spaces: for any pointed connected finite type simplicial set, the unit of the $\widetilde{\mathrm{L}}\!\dashv\!\widetilde{\mathrm{R}}$-adjunction, obtained from the 
$\mathrm{L}\!\dashv\!\mathrm{R}$-adjunction by removing artificial base points, 
 is homotopy equivalent to the Bousfield--Kan $\mathbb{Q}$-completion, see \cite[Theorem~7.19]{Robert-NicoudVallette20}. V. Roca i Lucio removed the pointed and connected assumptions by generalizing this theory to \emph{curved} complete $\sLi$-algebras in \cite{Roca24}. Notice that this work refines the methods in many fundamental ways: for instance, all the homotopical properties of the integration functor ($\infty$-groupoid, preservation of fibrations, homotopy invariance) are straightforward consequences of genuine new model category structures. 
\end{application}

\subsection{Absolute curved $\mathrm{EL}_\infty$-algebras}
Let us conclude this survey with the following question: is it possible to extend the (effective) integration theory of Lie type algebras beyond characteristic $0$? At first sight, since all the aforementioned formulas include rational coefficients with non-trivial denominators, there seems to be no hope in that direction. But once again, homotopy theory will lead us in the right direction. 

\medskip

In characteristic $0$, D. Sullivan \cite{Sullivan77} used commutative algebra structures on the cochains (or on the cohomology groups) of topological spaces to recover their rational homotopy type. The Koszul dual picture was given by D. Quillen  \cite{Quillen69} who rather considered Lie structures on the homotopy groups; a more direct approach is provided by the above integration theory of $\Li$-algebras, see also \cite{BFMT20}.
In characteristic $p>0$, the $p$-adic homotopy type of spaces was shown to by faithfully encoded in the $\mathrm{E}_\infty$-algebra structure of singular cochains by M. Mandell \cite{Mandell01}. Recall that $\mathrm{E}_\infty$ means ``everything'' homotopy commutative, where one relaxes both the associative relation \emph{and} the symmetry of the binary product, contrary to the abovementioned homotopy commutative structure on the cellular cochains $\mathrm{C}_\bullet$ of the standard geometric simplices where only the associative relation is relaxed up to homotopy. The following Koszul dual theory was recently developed by V. Roca i Lucio in \cite{Roca23}: it produces at the same time models for the $p$-adic homotopy type of spaces  from their homotopy groups and an (effective) integration theory in characteristic $p$. 

\medskip

Unlike the above integration theory of $\Li$-algebras, one cannot start anymore from the simplicial Dupont contraction to endow the 
cellular chains $\mathrm{C}^\bullet$ of the standard geometric simplices with a cosimplicial homotopy commutative coalgebra structure. 
The entry door to this new theory is the following canonical simplicial unital $\mathrm{E}_\infty$-coalgebra structure on $\mathrm{C}^\bullet$.

\begin{proposition}[\cite{BergerFresse04}]
The cellular chains $\mathrm{C}^\bullet$ of the standard geometric simplices admit a canonical 
cosimplicial unital $\mathcal{E}$-coalgebra structure, where $\mathcal{E}$
is the  Barratt--Eccles operad, whose arity $m$ component is made up of the bar construction of the symmetric group $\Sy_m$. 
\end{proposition}

There is a remarkable paradigm shift here:  one encodes \emph{coalgebra} structures with \emph{operads} whose operations are viewed upside down with one input and several outputs, i.e. in a prop way. This is mandatory here since cooperads encode \emph{conilpotent} coalgebras, property that is not satisfied in the present case. 
Dual structures should then be \emph{algebras} over a \emph{cooperad}, notion that already exists. 

\begin{definition}[Absolute $\mathcal{C}$-algebra]
For any cooperad $\mathcal{C}$, an \emph{absolute $\mathcal{C}$-algebra} is an algebra over the \emph{dual Schur monad} 
\[\gamma \colon \prod_{m\geqslant 0} \Hom_{\Sy_m}\left(\mathcal{C}(m), A^{\otimes m}\right)\mapsto A~. \]
\end{definition}

In an absolute algebra, the action of series of operations make sense without any prescribed complete topology. 
In the present case, one considers the Koszul dual curved cooperad given by  the bar construction $\mathrm{B} \mathcal{E}$ of the 
Barratt--Eccles operad. 

\begin{definition}[Absolute curved $\mathrm{EL}_\infty$-algebras]
An \emph{absolute curved $\mathrm{EL}_\infty$-algebra} is an algebra over the curved cooperad $\mathrm{B} \mathcal{E}$~. 
\end{definition}

By definition, this notion is a suitable homotopy generalization of Lie algebras in prime characteristic, which is an explicit 
point-set model for the (spectral) partition Lie algebras introduced in \cite{BM19}, see also \cite{BCN23}.
 In a way similar to $\mathrm{E}_\infty$-algebras, the letter ``$\mathrm{E}$'' in the chosen terminology ``$\mathrm{EL}_\infty$'' means that not only the Jacobi relation is relaxed up to homotopy, but also the skew-symmetry of the Lie bracket. Like in the above integration theory for $\Li$-algebras, the complete cobar construction of the cosimplicial unital $\mathcal{E}$-coalgebra on $\mathrm{C}^\bullet$ produces 
a \emph{universal} cosimplicial absolute curved $\mathrm{EL}_\infty$-algebra $\widehat{\Omega}_\iota \mathrm{C}^\bullet$~. 
Now we have all the key ingredients for the integration theory in prime characteristic.

\begin{theorem}[{\cite[Theorem~2.5]{Roca23}}]
There is a commuting diagram of Quillen adjunctions 
\[
\begin{tikzcd}[column sep=2pc,row sep=3pc]
&\hspace{1.5pc} \Omega \mathrm{B} \mathcal{E}\text{-}\mathsf{coalg} 
\arrow[dd, harpoon, shift left=1.1ex,  "\widehat{\Omega}_{\iota}"{name=F}] \arrow[ld, harpoon, shift left=.75ex, ""{name=C}]\\
\mathsf{sSet}  \arrow[ru, harpoon, shift left=1.5ex, "C"{name=A}]  \arrow[rd, harpoon, shift left=1ex, "\mathrm{L}"{name=B}] \arrow[phantom, from=A, to=C, , "\dashv" rotate=-65]
& \\
&\hspace{3pc}\mathsf{abs}~\mathsf{curv}~\mathrm{EL}_\infty\textsf{-}\mathsf{alg} ~, \arrow[uu, harpoon, shift left=.75ex, "\widehat{\mathrm{B}}_{\iota}"{name=U}] \arrow[lu, harpoon, shift left=.75ex, "\mathrm{R}"{name=D}] \arrow[phantom, from=B, to=D, , "\dashv" rotate=-114] \arrow[phantom, from=F, to=U, , "\dashv" rotate=-180]
\end{tikzcd}
\]
where $C$ stands for the cellular chain functor.
\end{theorem}

\begin{proof}[Sketch of proof]
The two left-most adjunctions with the category of simplicial sets are produced respectively by the 
cosimplicial unital $\mathcal{E}$-coalgebra structure on $\mathrm{C}^\bullet$, pulled back along the 
cofibrant resolution $\Omega \mathrm{B} \mathcal{E} \xrightarrow{\sim} \mathcal{E}$, and the cosimplicial absolute curved  $\mathrm{EL}_\infty$-algebra $\widehat{\Omega}_\iota \mathrm{C}^\bullet$~. A tour de force of \emph{op. cit.} lies in the achievement  of a complete bar construction $\widehat{\mathrm{B}}_\iota$ right-adjoint to the complete cobar construction: it relies on the intricate description of the not-necessarily conilpotent cofree coalgebra. 

One considers the classical model category structure of Kan--Quillen on simplicial sets, a canonical one on 
$\Omega \mathrm{B} \mathcal{E}$-coalgebras where weak equivalences are quasi-isomorphism, and where one transfers this latter model category structure onto 
absolute curved  $\mathrm{EL}_\infty$-algebras under the complete cobar construction, see \cite{LGRL23}. 
The category of $\Omega \mathrm{B} \mathcal{E}$-coalgebras
is homotopically better behaved than that of 
$\mathcal{E}$-coalgebras since $\Omega \mathrm{B} \mathcal{E}$ is a cofibrant operad. 
\end{proof}

The sole fact that the integration functor $\R$ is a right Quillen functor implies all the required homotopical properties: the image of the integration functor lies in $\infty$-groupoids and it sends respectively  epimorphisms to fibrations and filtered quasi-isomorphisms to weak equivalences \cite[Theorem~2.6]{Roca23}.

\begin{application}
This new integration theory of absolute curved $\mathrm{EL}_\infty$-algebras provides us with Lie type models for the $p$-adic homotopy theory: for connected finite type nilpotent simplicial sets, the unit of the $\mathrm{L} \dashv \R$-adjunction is a $\mathbb{F}_p$-equivalence \cite[Theorem~3.5]{Roca23}.
\end{application}

The integration theory of curved absolute $\mathrm{EL}_\infty$-algebras described above is not only fundamental, it is all constructive; we conclude this text with the explicit form of the main formulas. 
The set of planar rooted trees with vertices of positive arity $m\neq1$, labelled by strings $\left(\sigma_0, \ldots, \sigma_{r-1}\right)\in (\Sy_m)^{r}$ of permutations such that $\sigma_i\neq \sigma_{i+1}$, forms a basis for the curved cooperad $\mathrm{B} \mathcal{E}$. Therefore, this set parametrises the structural operations of  absolute curved  $\mathrm{EL}_\infty$-algebras. Among these labelled rooted trees, we denote the labelled corollas by $c_m^{\left(\sigma_0, \ldots, \sigma_{r-1}\right)}$~. 

\begin{definition}[Maurer--Cartan element]
A \emph{Maurer--Cartan element} of a curved absolute $\mathrm{EL}_\infty$-algebra 
$\g=\left(A, \d, \gamma\right)$ is a  degree $0$ element $\alpha\in A_0$ satisfying the \emph{Maurer--Cartan equation}: 
\[\d\alpha + \gamma\left(
\sum_{m\geqslant 0\,, \, m\neq 1} c_m^{(\id)}(\alpha, \ldots, \alpha)
\right)=0~.\]
\end{definition}

The first operations, corresponding to the trivially labelled corollas $c_m^{(\id)}$, of any absolute curved $\mathrm{EL}_\infty$-algebra form an absolute curved  $\mathrm{A}_\infty$-algebra: this structure alone produces the aforementioned Maurer--Cartan equation. As usual, we denote 
by $\MC(\g)$ the set of Maurer--Cartan elements. 

\medskip

The gauge equivalence between them uses more planar rooted trees than corollas. 
We consider the set $\mathsf{LaPRT}$ of planar rooted trees 
with leaves labelled by $01$ or $1$ and 
with all the vertices having non-negative arity $m\neq1$, being labelled by strings $\left(\sigma_0, \ldots, \sigma_{r-1}\right)\in (\Sy_m)^{r}$ of permutations such that $\sigma_i\neq \sigma_{i+1}$, and satisfying the following conditions: 
\begin{itemize}
\item[$\diamond$]  from left to right appear first the leaves $1, \ldots, a$ labelled by $01$, with $a\geqslant 1$, then the leaves $a+1, \ldots, a+b$  labelled by $0$, with $b\geqslant 0$, and finally the internal edges $a+b+1, \ldots, a+b+c=m$, with $c\geqslant 0$~, 

\item[$\diamond$] the string $\left(\sigma_0, \ldots, \sigma_{a-1}\right)\in (\Sy_m)^{a}$ labelling the vertex has $a$ permutations, 

\item[$\diamond$] $\sigma_{a-1}^{-1}(a+b+1)< \sigma_{a-1}^{-1}(a+b+2) <\cdots < \sigma_{a-1}^{-1}(m)$~, when $c\geqslant 1$~, 

\item[$\diamond$] for any $0\leqslant i \leqslant a-1$, the set $I_i\coloneq \sigma_i\left( \left\{ 1, \ldots, \sigma_i^{-1}(i+1)-1\right\}\right)$ is a subset of $\{a+1, \ldots, a+b\}$ and the sets $J_i\coloneq I_i\backslash \bigcup_{j=0}^{i-1} I_j$ form an ordered $J_0<\cdots< J_{a-1}$  partition of $\{a+1, \ldots, a+b\}$~.
\end{itemize}
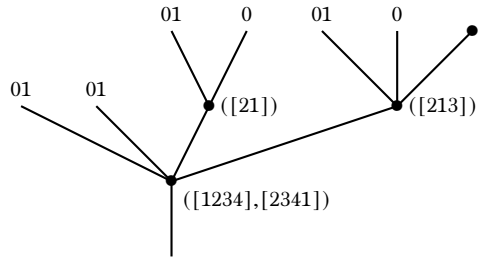
\begin{figure*}[h!]
	\[
\vcenter{\hbox{
\begin{tikzpicture}
	\coordinate (r) at (0,-1);
	\coordinate (v1) at (0,0);
	
	\coordinate (v2) at (-2, 1);	
	\coordinate (v3) at (-1,1);		
	\coordinate (v4) at (0.5,1);		
	\coordinate (v5) at (3,1);				
	
	\coordinate (v6) at (0, 2);		
	\coordinate (v7) at (1, 2);		
	\coordinate (v8) at (2, 2);		
	\coordinate (v9) at (3, 2);		
	\coordinate (v10) at (4, 2);		
					
	\draw[thick] (r) to (v1);
	\draw[thick] (v1) to (v2);	
	\draw[thick] (v1) to (v3);		
	\draw[thick] (v1) to (v4);	
	\draw[thick] (v1) to (v5);	
	\draw[thick] (v4) to (v6);				
	\draw[thick] (v4) to (v7);					
	\draw[thick] (v5) to (v8);						
	\draw[thick] (v5) to (v9);						
	\draw[thick] (v5) to (v10);								
			
	\node at (v1) {$\bullet$};
	\node at (v4) {$\bullet$};
	\node at (v5) {$\bullet$};		
	\node at (v10) {$\bullet$};			

	\node[above] at (v2) {$\scriptstyle 01$};
	\node[above] at (v3) {$\scriptstyle 01$};	
	\node[above] at (v6) {$\scriptstyle 01$};		
	\node[above] at (v8) {$\scriptstyle 01$};		
	\node[above] at (v7) {$\scriptstyle 0$};			
	\node[above] at (v9) {$\scriptstyle 0$};				

	\node[right] at (v4) {$\scriptstyle ([21])$};				
	\node[right] at (v5) {$\scriptstyle ([213])$};					
	\node[below right] at (v1) {$\scriptstyle ([1234], [2341])$};						
\end{tikzpicture}
	}}
	\]
	\caption{Example of a labelled planar rooted tree.}
	\label{Fig:PaPRT}
\end{figure*}

\begin{definition}[Gauge equivalence]
Two Maurer--Cartan elements $\alpha, \beta$ of a curved absolute $\mathrm{EL}_\infty$-algebra 
$\g=\left(A, \d, \gamma\right)$ are \emph{gauge equivalent} if there exists a \emph{gauge}  $\lambda\in A_1$ satisfying 
\[
\beta = \gamma\left(\sum_{\tau \in \mathsf{LaPRT}}
\tau(\lambda, \alpha ; \d \lambda + \alpha)~, 
\right)
\]
where 
$\tau(\lambda, \alpha ; \d \lambda + \alpha)$ is the image under the operation corresponding to the tree $\tau$ with leaves labelled by $01$ having input $\lambda$, leaves labelled by $1$ having input $\alpha$, and vertices of arity $0$ labelled by $\d \lambda+\alpha$. 
In this case, we use the notation $\alpha\sim_\lambda\beta$~.
\end{definition}

Notice that all the coefficients are equal to $1$, which justifies the relevance of this formula in prime characteristic. 

\begin{theorem}[{\cite[Theorem~2.14]{Roca23}}]
The image of any curved absolute $\mathrm{EL}_\infty$-algebra 
$\g=\left(A, \d, \gamma\right)$ under the integration functor satisfies 
\[\R(\g)_0 = \MC(\g)~,  \quad \R(\g)_1=\left\{(\alpha, \beta, \lambda) \in A_0^2\times A_1~, \alpha \sim_\lambda \beta\right\}~,  \quad and \quad \pi_0(\R(\g))\cong \mathcal{MC}(\g)~. \]
\end{theorem}

\begin{proof}[Sketch of proof]
This is a consequence of the respective $\mathcal{E}$-coalgebra structures on the cellular chains $C(\Delta^0)$ and $C(\Delta^1)$ given in \cite{BergerFresse04} and the 
the solution to the fixed-point equation associated to analytic functions given in \cite[Proposition~A.5]{Robert-NicoudVallette20}. For more details, we refer the reader to \emph{op. cit.}.
\end{proof}

\bibliographystyle{alpha}
\bibliography{bib}

\end{document}